\documentclass{amsart}
\date{September 10, 2013}
\usepackage[dvips]{graphicx} 
\usepackage{verbatim,enumerate}
\usepackage[usenames]{color}
\title[Zero mean curvature surfaces]{%
  Zero mean curvature 
  surfaces in \\
  Lorentz-Minkowski $3$-space and \\
  2-dimensional fluid mechanics
}

\author[Fujimori et.\ al.]{
        S.~Fujimori,  
	Y.~W.~Kim,    
	S.-E.~Koh,    
        W.~Rossman,   
	H.~Shin,      
        M.~Umehara,   
        K.~Yamada and 
	S.-D.~Yang}   

\address[Fujimori]{%
   Department of Mathematics,
   Faculty of Science, Okayama University,
   Okayama 700-8530, Japan}
\email{fujimori@math.okayama-u.ac.jp}
\address[Kim]{%
   Department of Mathematics,
   Korea University, Seoul 136-701, Korea
}
\email{ywkim@korea.ac.kr}
\address[Koh]{%
   Department of Mathematics,
   Konkuk University, Seoul 143-701, Korea}
\email{sekoh@konkuk.ac.kr}
\address[Rossman]{%
   Department of Mathematics,
   Faculty of Science,
   Kobe University,
   Kobe 657-8501, Japan
}
\email{wayne@math.kobe-u.ac.jp}
\address[Shin]{%
   Department of Mathematics,
   Chung-Ang University, Seoul 156-756, Korea
}

\email{hshin@cau.ac.kr}
\address[Umehara]{%
   Department of Mathematical and Computing Sciences,
   Tokyo Institute of Technology,
   Tokyo 152-8552, Japan
}
\email{umehara@is.titech.ac.jp}

\address[Yamada]{%
   Department of Mathematics,
   Tokyo Institute of Technology,
   Tokyo 152-8551, Japan
}
\email{kotaro@math.titech.ac.jp}
\address[Yang]{%
   Department of Mathematics,
   Korea University, Seoul 136-701, Korea
}
\email{sdyang@korea.ac.kr}

\subjclass[2010]{
  Primary 53A10;   
  Secondary 53B30. 
}

\keywords{%
    maximal surface, 
    minimal surface, 
    type change, 
    zero mean curvature}%
\usepackage{amsthm}
\theoremstyle{plain}
 \newtheorem{theorem}{Theorem}[section]
 \newtheorem{proposition}[theorem]{Proposition}
 \newtheorem{fact}[theorem]{Fact}
 \newtheorem{lemma}[theorem]{Lemma}
 \newtheorem{corollary}[theorem]{Corollary}
 \newtheorem*{problem}{Problem}
\theoremstyle{definition}
 \newtheorem{definition}[theorem]{Definition}
\theoremstyle{remark}
 \newtheorem{remark}[theorem]{Remark}
 \newtheorem*{remark*}{Remark}
 
 \newtheorem*{acknowledgement}{Acknowledgement}
\numberwithin{equation}{section}
\renewcommand{\theenumi}{{\rm(\arabic{enumi})}}
\renewcommand{\labelenumi}{\theenumi}
\renewcommand{\Re}{\operatorname{Re}}
\renewcommand{\Im}{\operatorname{Im}}
\newcommand{\II}{I\!I}
\newcommand{\III}{I\!I\!I}
\newcommand{\vect}[1]{\boldsymbol{#1}}
\newcommand{\R}{\boldsymbol{R}}
\newcommand{\C}{\boldsymbol{C}}
\newcommand{\Z}{\boldsymbol{Z}}
\renewcommand{\phi}{\varphi}
\newcommand{\inner}[2]{\left\langle{#1},{#2}\right\rangle}

\newcommand{\Cat}{\mathcal{C}}
\newcommand{\Sch}{\mathcal{S}}

\newcommand{\pmt}[1]{{\begin{pmatrix} #1  \end{pmatrix}}}
\newcommand{\rot}{\operatorname{rot}}
\renewcommand{\div}{\operatorname{div}}
\newcommand{\grad}{\operatorname{grad}}
\begin{document}
\begin{abstract}
 Space-like maximal 
 surfaces  and time-like minimal surfaces 
 in Lorentz-Minkowski $3$-space  $\R^3_1$  
 are both characterized as zero mean 
 curvature surfaces.
 We are interested in the case where 
 the 
 zero mean 
 curvature
 surface changes type from
 space-like to time-like at a
 given non-degenerate null curve.
 We consider this phenomenon and
 its interesting connection
 to $2$-dimensional fluid mechanics
in this expository article.
\end{abstract}
\thanks{
  Kim was supported by NRF 2009-0086794, 
  Koh by NRF 2009-0086794 and NRF 2011-0001565, 
  and Yang by NRF 2012R1A1A2042530.
  Fujimori was partially supported by the Grant-in-Aid for 
  Young Scientists (B) No.\ 21740052,
  Rossman was supported by 
  Grant-in-Aid for Scientific Research (B) No.\ 20340012, 
  Umehara by (A) No.\ 22244006 and 
  Yamada by (B) No.\ 21340016
  from the Japan Society for the Promotion of Science.
}
\maketitle

\section{Introduction}
We denote by $\R^3_1:=\{(t,x,y)\,;\,t\in\R\}$ the Lorentz-Minkowski
$3$-space with the metric $\inner{~}{~}$ of signature $(-,+,+)$.
 Space-like maximal  surfaces  and time-like minimal surfaces 
 in Lorentz-Minkowski $3$-space $\R^3_1$  
 are both characterized as zero mean 
 curvature surfaces.
This is an expository article about 
type changes of zero mean curvature surfaces
in $\R^3_1$. 
Klyachin \cite{Kl} showed,
under a sufficiently weak regularity assumption, 
that a zero mean curvature surface
in $\R^3_1$ changes its causal type only 
on the following two subsets:
\begin{itemize} 
\item null curves
      (i.e., regular curves whose velocity vector fields are light-like)
      which are non-degenerate (cf.\ Definition~\ref{def:n-deg}), or
 \item light-like lines, which are degenerate everywhere.
\end{itemize} 
Recently, actual occurrence
of the second case was shown
in the authors' previous work \cite{Null}.
So we now pay attention to the former possibilities:
Given a non-degenerate null curve $\gamma$ in $\R^3_1$, 
there exists a zero mean curvature surface which changes 
its causal type across this curve from a space-like 
maximal surface to a time-like minimal surface 
(cf.\ \cite{G,G1,G2}, \cite{Kl}, \cite{KY2} and \cite{KKSY}).
This construction can be accomplished using 
the Bj\"orling formula for the Weierstrass-type
representation formula of maximal surfaces.
By unifying the results of Gu \cite{G}, Klyachin \cite{Kl}, and
\cite{KKSY},
we explain the mechanism for how zero mean curvature surfaces 
change type across non-degenerate null curves,
and give
\lq the fundamental theorem
of type change for zero mean curvature surfaces\rq\
(cf.\ Theorem \ref{thm:main2})
in the second section of this paper.
Locally, such a surface is the
graph of a function $t=f(x,y)$ satisfying 
\begin{equation}\label{zm}
    (1-f_y^2)f_{xx} + 2 f_x f_yf_{xy}+(1- f_x^2)f_{yy}=0.
\end{equation}
We call this and its graph the {\it zero mean curvature equation\/} 
and a {\it zero mean curvature surface} or 
{\it zero mean curvature graph}, respectively.

\begin{figure}[t]%
 \begin{center}
  \begin{tabular}{c@{\hspace{2em}}c@{\hspace{2em}}c@{\hspace{2em}}c}
       \includegraphics[width=4.5cm]{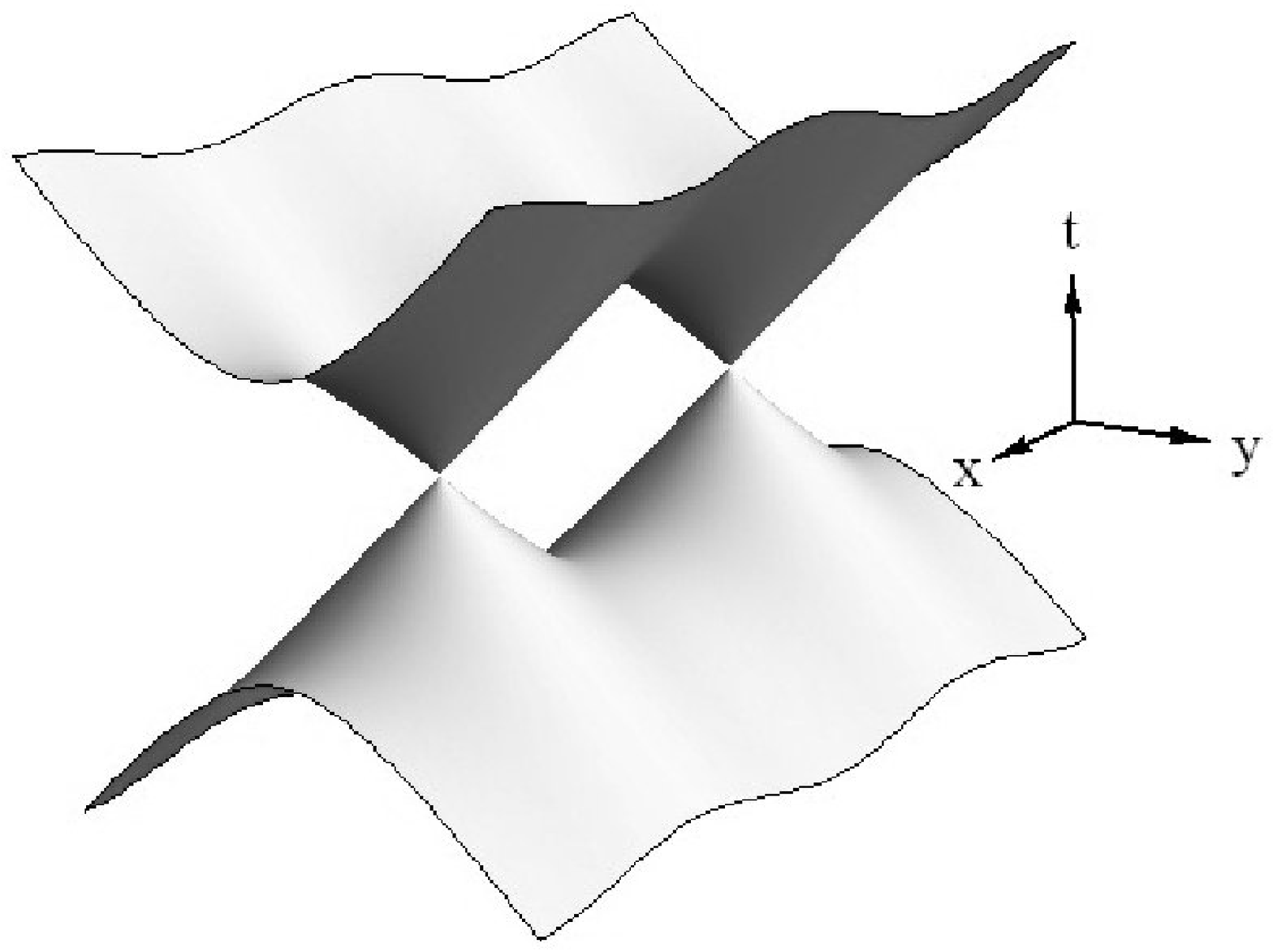} & 
       \includegraphics[width=4.5cm]{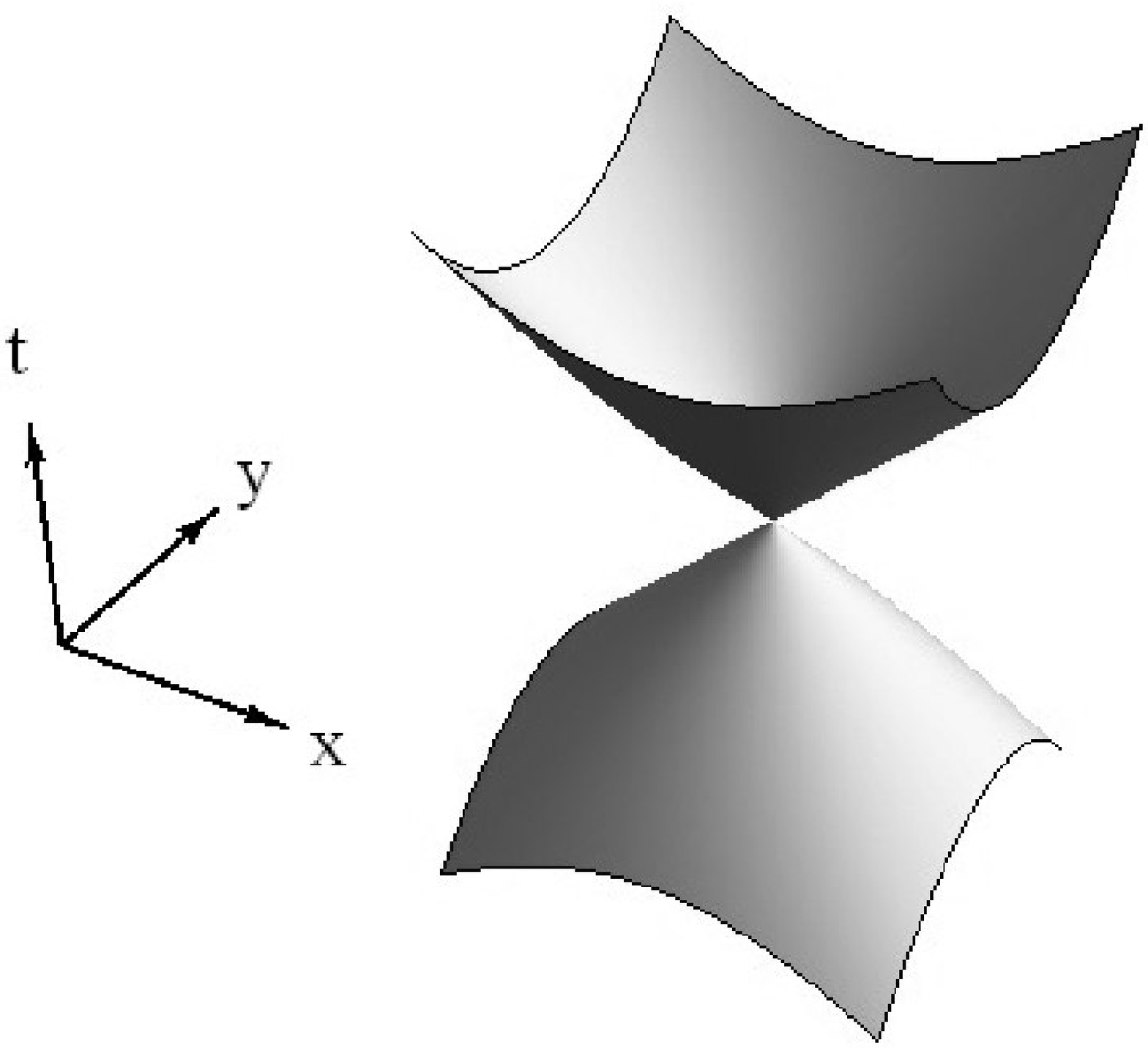} 
  \end{tabular}
 \end{center}
 \caption{%
  Hyperbolic catenoids $\Cat_+$, $\Cat_-$.
 }%
\label{fig1a}
\end{figure}

As pointed out in \cite{CR},
the space-like hyperbolic catenoid
\begin{equation}\label{eq:S}
  \Cat_+=\{(t,x,y)\in \R^3_1\,;\, \sin^2 x+y^2-t^2=0\}
\end{equation}
and the time-like hyperbolic catenoid
\begin{equation}\label{eq:T}
  \Cat_-=\{(t,x,y)\in \R^3_1\,;\, \sinh^2 x+y^2-t^2=0\}
\end{equation}
are both typical examples of zero mean curvature
surfaces containing singular light-like lines as 
subsets (cf.\ Figure \ref{fig1a}).
The space-like hyperbolic catenoid $\Cat_+$ is singly periodic. 

Also, both the space-like Scherk surface (cf.\ \cite{CR})
\begin{equation}\label{eq:Sp}
  \Sch_+=\{(t,x,y)\in \R^3_1\,;\, \cos t =\cos x \cos y\}
\end{equation}
and the time-like Scherk surface of the first kind 
(cf.\ \cite{CR})
\begin{equation}\label{eq:St}
  \Sch_-=\{(t,x,y)\in \R^3_1\,;\, \cosh t=\cosh x \cosh y\}
 \end{equation}
contain singular light-like lines as 
subsets (cf.\ Figure \ref{fig1b}).
As seen in the left-hand side of
Figure~\ref{fig1b}, 
$\Sch_+$ is triply periodic.

As an application of the results in Section~\ref{sec:fundamental},
we show in Section~\ref{sec:conjugate} that
$\Cat_+$ and $\Cat_-$
induce a common zero mean curvature graph
 (cf.\ Figure \ref{fig1c}, left)
\begin{equation}\label{eq:K1}
  \Cat_0=\{(t,x,y)\in \R^3_1\,;\, t=y \tanh x\}
\end{equation}
via their conjugate surfaces.
The graph $\Cat_0$ changes type at two non-degenerate null curves.
Similarly, we also show that the Scherk-type surfaces $\Sch_+$
and $\Sch_-$ induce a common zero mean curvature graph
via their conjugate surfaces  (cf.\ Figure \ref{fig1c}, right)
\begin{equation}\label{eq:K2}
  \Sch_0=\{(t,x,y)\in \R^3_1\,;\, 
  e^t \cosh x=\cosh y 
  \},
\end{equation}
which changes type at four non-degenerate
null curves.
These two phenomena were 
briefly commented upon
in \cite{CR}.
The entire zero-mean curvature graphs
$\Cat_0$ and $\Sch_0$ were discovered
by Osamu Kobayashi \cite{K}.
On the other hand, the above three examples
\eqref{eq:Sp}, \eqref{eq:St} and \eqref{eq:K2} are
particular cases of the general families presented in
Sergienko and Tkachev \cite[Theorem 2]{ST}. 
Moreover, several doubly periodic mixed type 
zero mean curvature graphs with 
isolated singularities are given in \cite{ST}.
Space-like maximal surfaces
frequently have singularities.
See the references
\cite{ER}, \cite{UY} and \cite{FSUY}
for general treatment of these singularities.

\begin{figure}[htb]%
 \begin{center}
  \begin{tabular}{c@{\hspace{2em}}c}
       \includegraphics[height=3.8cm]{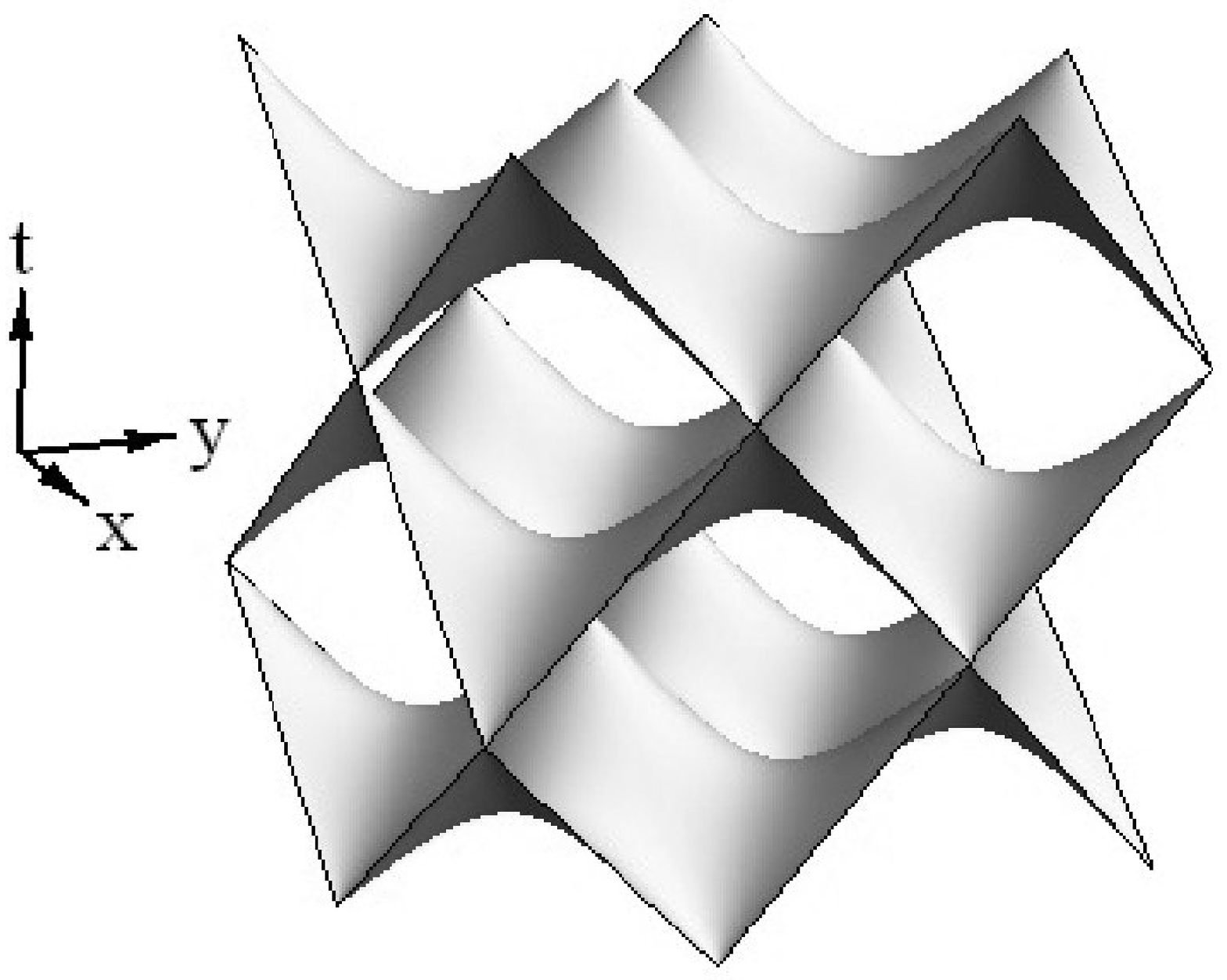} & 
       \includegraphics[height=4.5cm]{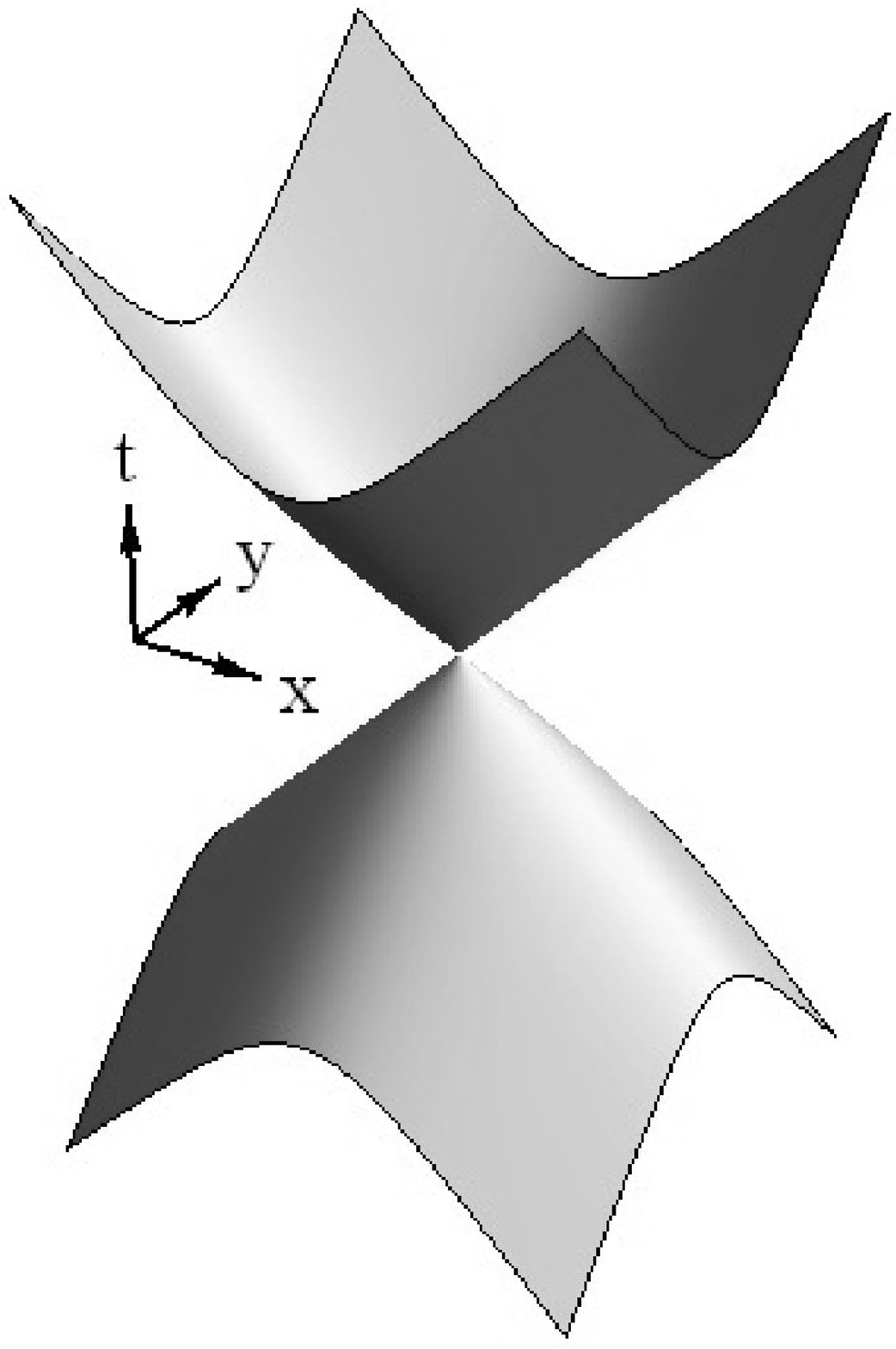}  
  \end{tabular}
 \end{center}
\caption{%
    Scherk-type surfaces $\Sch_+$ and $\Sch_-$.
}%
\label{fig1b}
\end{figure}

In Section \ref{sec:fluid}, we remark on an interesting connection
between zero mean curvature surfaces in $\R^3_1$ and 
irrotational two-dimensional barotropic steady flow,
where
the fluid is called {\em barotropic}
if
the pressure $p$ is a function depending 
only on the density $\rho$.
In fact, the stream function $\psi(x,y)$ satisfies
(cf.\ \cite{R}, 
see also Proposition~\ref{prop:eq} in Section~\ref{sec:fluid}%
)
\begin{equation}\label{eq:st}
    (\rho^2c^2-\psi_y^2)\psi_{xx} + 2 \psi_x \psi_y\psi_{xy}+
    (\rho^2c^2- \psi_x^2)\psi_{yy}=0,
\end{equation}
where $c$ is the local speed of sound
given by $c^2=dp/d\rho$
 (cf.~\eqref{eq:sound-speed}).
We choose the units so that $\rho=1$ and 
$c=1$ when $\psi_x=\psi_y=0$.
Then the product $\rho c$ is equal to $1$ if 
\begin{equation}\label{eq:virtual-pressure}
  p =p_0-\frac{1}{\rho},
\end{equation}
for a constant $p_0$,
which implies that the graphs of zero mean curvature surfaces 
can be interpreted as stream functions of a {\it virtual gas} 
with \eqref{eq:virtual-pressure}.
(In fact, $p$ is approximately proportional 
to $\rho^{1.4}$ for air.)
As an application of the singular  Bj\"orling formula
for zero mean curvature surfaces
(cf.\ Theorem~\ref{thm:main2}), 
we can construct a family
of stream functions which change from being 
subsonic to supersonic at a given locally convex curve in the $xy$-plane.
The velocity vector fields of these gas flows 
diverge at the convex
curve, although the streaming functions are real analytic.

\begin{figure}[t]%
 \begin{center}
  \begin{tabular}{c@{\hspace{2em}}c@{\hspace{2em}}c@{\hspace{2em}}c}
       \includegraphics[width=4.5cm]{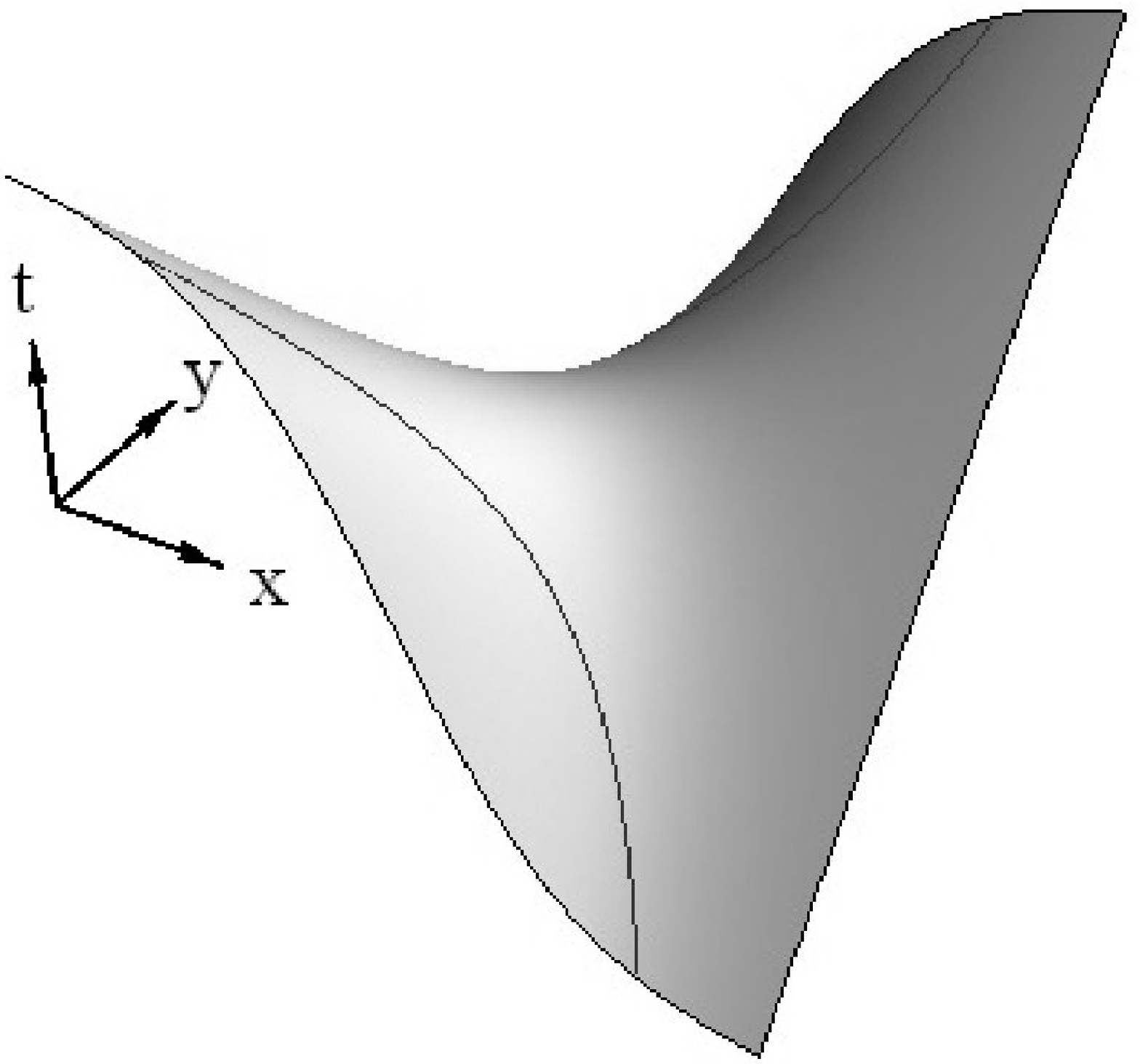} & 
       \includegraphics[width=4.8cm]{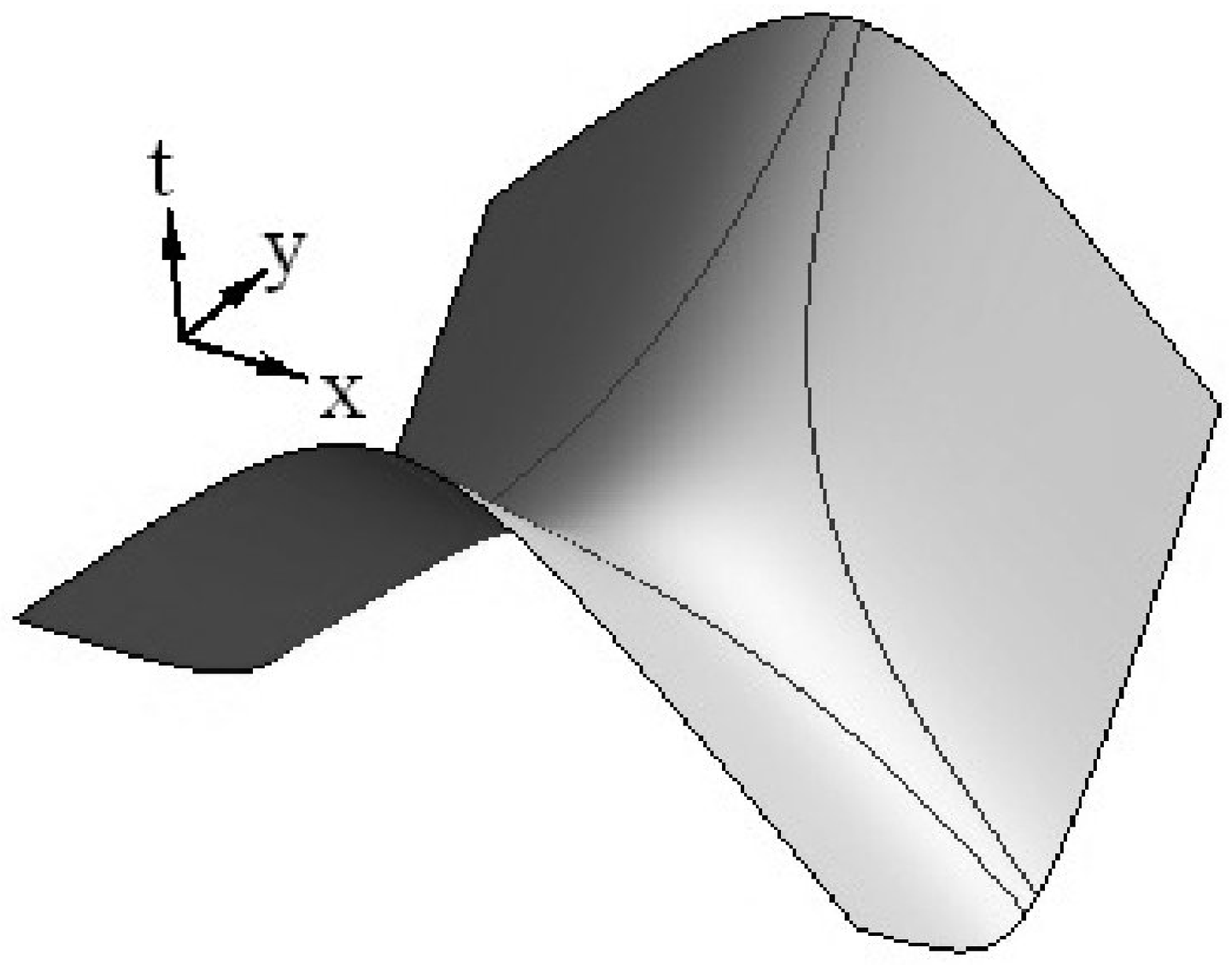} 
  \end{tabular}
 \end{center}
\caption{%
 Zero mean curvature graphs  $\Cat_0$ and $\Sch_0$
 (the curves where the surfaces change type are also indicated).
}%
\label{fig1c}
\end{figure}

\section{%
 Type change of zero mean curvature surfaces
}
\label{sec:fundamental}

In this section, we discuss type change 
for zero mean curvature surfaces,
by unifying the results of Gu \cite{G,G1, G2}, Klyachin \cite{Kl} and 
four of the authors here \cite{KKSY}.

A regular curve $\gamma:(a,b)\to \R^3_1$
is called {\it null\/} or {\it isotropic\/}
if $\gamma'(t):=d\gamma(t)/dt$ is a light-like vector for all $t\in
(a,b)$.
\begin{definition}\label{def:n-deg}
 A null curve  $\gamma:(a,b)\to \R^3_1$
 is called {\it degenerate} or {\it non-degenerate\/} at $t=c$
 if $\gamma''(c)$ is or is not proportional
 to the velocity vector $\gamma'(c)$, respectively.
 If $\gamma$ is non-degenerate at each $t\in (a,b)$,
 it is called a {\it non-degenerate\/} null curve.
\end{definition}
We now give a characterization of
zero mean curvature surfaces that change type across
a real analytic non-degenerate null curve.
Given an arbitrary real analytic 
null curve $\gamma : (a,b) \rightarrow \R^3_1$, 
we denote the unique analytic extension of 
it by $\gamma$ again throughout this article, 
by a slight abuse of notation.
We consider the two surfaces
\begin{align*}
   \Phi(u,v)&:=\frac{\gamma(u+iv)+\gamma(u-iv)}2,
   \intertext{and}
   \Psi(u,v)&:=\frac{\gamma(u+v)+\gamma(u-v)}2,
\end{align*}
which are defined for $v$ sufficiently close to zero. 
We recall the following assertion:
\begin{proposition}[\cite{G,G1,G2}, \cite{Kl} and \cite{KKSY}]
\label{thm:KKSY}
 Given a real analytic non-degenerate null curve 
 $\gamma:(a,b)\to \R^3_1$, the union of the images of $\Phi$ and 
 $\Psi$ given as above
 are subsets of a real analytic
 immersion, and the intersection is $\gamma$.
 Moreover, $\Phi$ gives a space-like maximal surface
 and $\Psi$ gives a time-like minimal surface if 
 $v$ is sufficiently close to zero.
 Furthermore,  this analytic extension of
 the curve $\gamma$ as a zero mean curvature surface
 does not depend upon the choice of the real analytic
 parametrization of the curve $\gamma$.
\end{proposition}
\begin{proof}
 We give here a proof for the sake of the readers' convenience.
 We have that  
 \[
   \Phi(u,v)=\sum_{n=0}^\infty (-1)^n \gamma^{(2n)}(u)v^{2n},\qquad
   \Psi(u,v)=\sum_{n=0}^\infty \gamma^{(2n)}(u)v^{2n}
 \]
 near $v=0$, where
 $\gamma^{(j)}={d^j \gamma}/{d t^j}$.
 In particular, if we set
 $H(u,v):=\sum_{n=0}^\infty \gamma^{(2n)}(u)v^{n}$,
 then it gives a germ of a real analytic function satisfying
 \[
    H(u,-v^2)=\Phi(u,v),\qquad H(u,v^2)=\Psi(u,v),
 \]
 which prove that 
 the images of $\Phi$ and $\Psi$ lie on a 
 common real analytic surface.
 Since $\gamma$ is non-degenerate, the two vectors
 \[
    H_u(u,0)=\gamma'(u),\qquad H_v(u,0)=\gamma''(u)
 \]
 are linearly independent, and $H$ gives 
 an immersion which contains $\gamma$.

 Moreover, it can be easily checked that
 $\Phi$ gives a space-like maximal 
 surface (cf.\ Lemma \ref{lem:fold})
 and $\Psi$ gives a  time-like minimal  surface.

 We now show the last assertion:
 Since the surface is real analytic, 
 it is sufficient to show that
 given an arbitrary real analytic diffeomorphism $f$ from $(a,b)$ 
 onto its image in $\R$,
 \[
    \Psi(u,v)=\frac{\gamma(u+iv)+\gamma(u-iv)}2 \qquad
    \text{and}\qquad
    \tilde \Psi(u,v)=\frac{\tilde\gamma(u+iv)+\tilde\gamma(u-iv)}2
 \]
 induce the same surface as their graphs,
 where $\tilde \gamma(t):=\gamma(f(t))$.
 We define $A$, $B$ by
 \[
     A=(f(u+v)+f(u-v))/2,\qquad B=(f(u+v)-f(u-v))/2.
 \]
 Thus it is sufficient to show that
 the map
 \[
    (u,v)\mapsto (A,B)
 \]
 is an immersion at $(u,0)$. 
 In fact, the Jacobian of the map is given by
 \[
    J=\det\pmt{f'(u) & 0 \\ 0 & f'(u)}\ne 0.
 \]
\end{proof}

\begin{definition}
\label{def:nondegtc}
 Let $\Omega^2$ be a domain in $\R^2$
 and $f:\Omega^2\to \R$ a $C^\infty$-function satisfying \eqref{zm}.
 We set
 \[
    B:=1-f_x^2-f_y^2.
 \]
 A point $p$ on $\Omega^2$ is called a
 {\it non-degenerate point of  type change\/}%
\footnote{%
    In Gu \cite{G1},
    \lq dual regularity\rq\ for points of type change
    is equivalent to our notion. 
    Klyachin \cite{Kl}  did not define this particular notion, 
    but use it in an essential way.
} %
 with respect to $f$ if 
 \[
     B(p)=0,\qquad \nabla B(p)\ne0,
 \]
 where $\nabla B:=(B_x,B_y)$ is the
 gradient vector of the function $B$.
\end{definition}

Since $\nabla B$ does not vanish at $p$, 
the function $f$ actually changes type at the non-degenerate 
point $p$.
\begin{proposition}[\cite{G1,G2}]
\label{prop:equiv}
 Under the assumption that  $B(p)$ vanishes,
 the following two assertions are equivalent.
 \begin{enumerate}
  \item\label{item:equiv:1} 
       $p$ is a non-degenerate point of  type change.
  \item\label{item:equiv:2}
        $p$ is a dually regular point in the sense of 
	\cite{G1}, that is, $f_{xx}f_{yy}-(f_{xy})^2$ 
	does not vanish at $p$.
 \end{enumerate}
\end{proposition}
\begin{proof}
 Note that $ (\nabla B)^T = -2 H (\nabla f)^T$,
 where ${}^T$ is the transpose and
 $H:=\pmt{f_{xx} & f_{xy} \\ f_{xy} & f_{yy}}$.
 Note also that $B(p)=0$ implies that 
 $\nabla f(p) \not= \vec{0}$.

 Now suppose that \ref{item:equiv:2} holds. 
 Then $\det{H}(p) \neq 0$, which with 
 $\nabla f(p) \not= \vec{0}$ implies that 
 $H(p)(\nabla f(p))^T \neq \vec{0}$, that is,
  \ref{item:equiv:1}
 holds.

 Suppose on the contrary that \ref{item:equiv:2} does not hold. 
 By a suitable linear coordinate change of $(x,y)$, 
 we may assume without loss of generality that $f_{xy}(p)=0$. 
 Then either $f_{xx}(p)=0$ or $f_{yy}(p)=0$.
 Also, the zero mean curvature equation 
 \[
    0=(1-f_y^2)f_{xx}+2f_xf_yf_{xy}+(1-f_x^2)f_{yy}
 \]
 with $B(p)=0$ and $f_{xy}(p)=0$ implies that 
 \[
    f_x(p)^2 f_{xx}(p)+f_y(p)^2f_{yy}(p)=0.
 \]
 This with $f_{xx}(p)=0$ or $f_{yy}(p)=0$ implies that 
 \[
    H(p)(\nabla f(p))^T=
      \pmt{f_x(p) f_{xx}(p) \\ f_y(p) f_{yy}(p)} = 
      \pmt{0 \\ 0},
 \]
 so \ref{item:equiv:1} does not hold. 
\end{proof}

Moreover, the following assertion holds:
\begin{proposition}[\cite{G1,G2}, \cite{Kl}]
 \label{prop:Kl1}
 Let $\gamma$ be a real analytic non-degenerate null
 curve, and let $f_\gamma$ be the real analytic function
 induced by $\gamma$ as in Proposition \ref{thm:KKSY}, which satisfies \eqref{zm}.
 Then the image of $\gamma$
 consists of non-degenerate points of  type change 
 with respect to $f_\gamma$.
\end{proposition}
Note that the conclusion is stronger than that of 
Proposition~\ref{thm:KKSY}.
\begin{proof}
 Let $\gamma$ be a non-degenerate null curve.
 Without  loss of generality, we may 
 take the time-component $t$ as a parametrization of $\gamma$.
 Then we have the expression
 \[
    \gamma(t)=(t, x(t),y(t))\qquad (a<t<b)
 \]
 such that
 \begin{equation}\label{eq:a1}
  x'(t)^2+y'(t)^2=1.
 \end{equation}
 Since $\gamma$ is non-degenerate, it holds that
 \begin{equation}\label{eq:a2}
  0\ne \gamma''(t)=(0,x''(t),y''(t)).
 \end{equation}
 Differentiating the relation $t=f(x(t),y(t))$,
 we have that
 \begin{equation}\label{eq:a3}
  x'(t)f_x(x(t),y(t))+y'(t)f_y(x(t),y(t))=1.
 \end{equation}
 On the other hand, 
 the relation $B=0$ implies that
 \begin{equation}\label{eq:a4}
  f_x(x(t),y(t))^2+f_y(x(t),y(t))^2=1.
 \end{equation}
 Then by \eqref{eq:a1}, \eqref{eq:a3} and \eqref{eq:a4},
 it holds that
 \[
    x'(t)=f_x,\qquad y'(t)=f_y.
 \]
 Thus we have that
 \begin{align}
  \label{eq:id}
  (x'',y'')&=\frac{d}{dt}\biggl(f_x(x(t),y(t)),
                    f_y(x(t),y(t))\biggr)\\
           &=(x'f_{xx}+y'f_{xy},x'f_{xy}+y'f_{yy}) \nonumber
        \\ &=(f_xf_{xx}+f_yf_{xy},
           f_xf_{xy}+f_yf_{yy})=-\frac12\nabla B.
           \nonumber
 \end{align}
By \eqref{eq:a2}, we get the assertion.
\end{proof}
Conversely, we can prove the following.
\begin{proposition}[{\cite{G1,G2}, \cite[Lemma 2]{Kl}}]
 \label{prop:Kl2}
 Let $f:\Omega^2\to \R$ be a  $C^\infty$-function
 satisfying the zero mean curvature equation \eqref{zm},
 and let $p=(x_0,y_0)\in \Omega^2$ be a
 non-degenerate point of  type change.
 Then there exists a non-degenerate
 $C^\infty$-regular
 null curve in $\R^3_1$ with image passing 
 through $(f(x_0,y_0),x_0,y_0)$
 and contained in the graph of $f$.
\end{proposition}
\begin{proof}
 By the implicit function theorem,
 there exists a unique $C^\infty$-regular curve 
 $\sigma(t) =(x(t),y(t))$ in the $xy$-plane with $p=\sigma(0)$ so 
 that $B=0$ along the curve.
 Since $B=0$ on $\sigma$, the velocity vector $\sigma'$ is 
 perpendicular to $\nabla B$.
 Since $\nabla f$ is also perpendicular to $\nabla B$,
 we can conclude that $\nabla f$ is proportional to
 $\sigma'$. In fact
 \begin{align*}
    -\frac12\nabla f\cdot \nabla B
      &=(f_x,f_y)
          \pmt{f_xf_{xx}+f_yf_{xy}\\f_xf_{xy}+f_yf_{yy}}
       = f_x^2f_{xx}+2f_xf_yf_{xy}+f_y^2f_{yy}\\
      &= (1-f_y^2)f_{xx}+2f_xf_yf_{xy}+(1-f_x^2)f_{yy}
            -(1-f_x^2-f_y^2)(f_{xx}+f_{yy})\\
     &=0.
 \end{align*}
 Since $f_x^2+f_y^2=1$, by taking an arclength
 parameter of $\sigma$, we may set
 \[
     x'=f_x,\qquad y'=f_y,
 \]
 and then
 \[
     B=1-f_x^2-f_y^2=1-(x')^2-(y')^2=0
 \]
 holds along $\sigma$,
 which implies that $t\mapsto (t,x(t),y(t))$ is a null curve.
 Since
 \[
    \frac{d}{dt}f(x(t),y(t))=x'f_x+y'f_y=f_x^2+f_y^2=1,
 \]
 there exists a constant $c$ such that
 $f(x(t),y(t))=t+c$. By translating the graph vertically if necessary,
 we may assume that
 \[
     f(x(t),y(t))=t
 \]
 holds for each $t$.
 Then we obtain the identity \eqref{eq:id}
 in this situation, which
 implies that $\nabla B(p)=(x''(0),y''(0))\ne 0$,
 namely, 
 \[
    (a,b)\ni t\mapsto (f(x(t),y(t)),x(t),y(t))=(t,x(t),y(t))\in \R^3_1
 \]
 gives a non-degenerate null curve near $t=0$
 lying in the graph of $f$.
\end{proof}
\begin{definition}[\cite{Extended}]
\label{def:maxface}
 Let $\Sigma^2$ be a Riemann surface.
 A $C^\infty$-map $\phi:\Sigma^2\to \R^3_1$
 is called a {\it  generalized maximal surface\/}
 if there exists an open dense subset $W$ of $\Sigma^2$
 such that the restriction $\phi|_W$ of $\phi$ 
 to $W$ gives a conformal (space-like) 
 immersion of zero mean curvature.
 A {\it singular point\/} of $\phi$ is a point at which
 $\phi$ is not an immersion.
 A singular point $p$ satisfying $d\phi(p)=0$
 is called a {\it branch point\/}  of $\phi$.  
 Moreover, $\phi$ is called a {\it maxface\/}
 if $\phi$ does not have any branch points.
 (A maxface may have singular points in general).
\end{definition}
\begin{remark}
 The above definition of maxfaces is 
 given in \cite{Extended}, which is 
 simpler than the definition
 given in \cite{UY} and \cite{FSUY}. 
 However, this
 new definition is equivalent to the previous one,
 as we now explain.
 Suppose that $\phi|_W$ is a conformal (space-like) 
 immersion of zero mean curvature.
 Then $\partial \phi=\phi_z\,dz$ is a $\C^3$-valued
 holomorphic 1-form on $W$, 
where $z$ is a complex coordinate of $\Sigma^2$.
 Since $\phi$ is a $C^\infty$-map on $\Sigma^2$,
 $\partial \phi$ can be  holomorphically extended to $\Sigma^2$.
 Then the line integral 
 $\Phi(z)=\int_{z_0}^z  \partial \phi$
 with respect to a base point $z_0\in \Sigma^2$
 gives a holomorphic map defined on 
 the universal cover of $\Sigma^2$ whose real part 
 coincides with $\phi(z)-\phi(z_0)$.
 The condition that $\phi$ does not have any branch point
 implies that $\Phi$ is an immersion. 
 Moreover, since $\phi$ is conformal on $W$, 
 $\Phi$ satisfies
 \[
    -(d\Phi_0)^2+(d\Phi_1)^2+(d\Phi_2)^2=0\qquad
    \bigl(\Phi=(\Phi_0,\Phi_1,\Phi_2)\bigr),
 \]
 namely, $\Phi$ is a null immersion.
 So $\phi$ satisfies the definition of maxface as in
 \cite{UY} and \cite{FSUY}.
 We call $\Phi$ the {\em holomorphic lift\/} of 
 the maxface $\phi$.
\end{remark}
\begin{remark}
 By the above definition, maxfaces
 are orientable. However, there are  
 non-orientable maximal surfaces,
 as shown in \cite{FL}.
 The definition of non-orientable maxfaces
 is given in \cite[Def.\ 2.1]{FL}.
 In this paper, we work only with orientable maximal
 surfaces.
 It should be remarked that non-orientable
 maxfaces will be orientable
 when taking double coverings. 
\end{remark}
Let $\phi:\Sigma^2\to \R^3_1$
be a maxface with Weierstrass data $(G,\eta)$ 
(see \cite{UY} for the definition of Weierstrass data).
Using the data $(G,\eta)$, the maxface $\phi$ has the expression
\begin{equation}\label{eq:Wei}
  \phi=\Re(\Phi),\quad
  \Phi=\int_{z_0}^z (-2G, 1+G^2,i(1-G^2))\eta.
\end{equation}
The imaginary part
\begin{equation}\label{eq:Conj}
  \phi^*:=\Im(\Phi):\tilde \Sigma^2\longrightarrow \R^3_1
\end{equation}
also gives a maxface called the {\it conjugate surface\/} of $\phi$,
which is defined on the universal cover $\tilde \Sigma^2$ of $\Sigma^2$.
The following fact is known:
\begin{fact}[\cite{UY,FSUY}]
\label{fact:sing}
 A point $p$ of $\Sigma^2$ is  a singular point of $\phi$
 if and only if $|G(p)|=1$.
\end{fact}

\begin{definition}\label{def:dg}
 A singular point $p$ of $\phi$ is called  
 {\it non-degenerate\/} if $dG$ does not vanish at $p$.
\end{definition}

\begin{fact}[\cite{UY,FSUY}]\label{fact:non-d}
 If a singular point $p$ of $\phi$ is 
 non-degenerate, then there exists a 
 neighborhood $U$ of $p$ and a regular curve 
 $\gamma(t)$ in $U$ so that $\gamma(0)=p$ and
 the singular set of $\phi$ in $U$
 coincides with the image of the curve $\gamma$.
\end{fact}
This curve $\gamma$ is called the {\it singular curve\/}
at the non-degenerate singular point $p$.

\begin{definition}\label{def:fold}
 A regular curve $\gamma$ on $\Sigma^2$ is called
 a {\it non-degenerate  fold singularity\/}
 if it consists of non-degenerate singular points
 such that the real part of
 the meromorphic function $dG/(G^2\eta)$ vanishes identically
 along the singular curve $\gamma$.
 Each point on the non-degenerate  fold singularity
 is called a {\it fold singular point}.
\end{definition}
A singular point of $C^\infty$-map $\phi:\Sigma^2\to \R^3$
has a {\it fold singularity} at $p$ if there exists a
local coordinate $(u,v)$ centered at $p$ such that
$\phi(u,v)=\phi(u,-v)$. 
Later, we show that a non-degenerate  fold singularity is 
actually a fold singularity (cf.\ Lemma~\ref{lem:max-ext}).

Suppose that $p$ is  a non-degenerate
fold singular point of $\phi$. 
The following  duality between fold singularities and generalized 
cone-like singularities (cf.\ \cite[Definition 2.1]{FRUYY}) holds:
\begin{proposition}[\cite{KY2}]\label{prop:dual}
 Let $\phi:\Sigma^2\to \R^3_1$
 be a maxface and $\phi^*$ the conjugate maxface. 
 Then  $p$ is a non-degenerate fold singular point
 of $\phi$ if and only if
 it is a generalized cone-like singular point
 of $\phi^*$.
\end{proposition}
\begin{proof}
 This assertion is immediate from comparison of
 the above definition of non-degenerate
 fold singularity and the definition of
 generalized cone-like singular points as in \cite[Definition 2.1
 and Lemma~2.3]{FRUYY}.
\end{proof}

We now show the following assertion,
which characterizes the non-degenerate fold singularities on maxfaces.

\begin{theorem}\label{thm:main}
 Let $\phi:\Sigma^2\to \R^3_1$ be a maxface
 which has non-degenerate fold singularities 
 along a singular curve $\gamma:(a,b) \rightarrow \Sigma^2$.
 Then $\hat\gamma:=\phi\circ \gamma$ is a 
 non-degenerate null curve,
 and the image of the map
 \begin{equation}\label{eq:Bj}
   \tilde \phi(u,v):=\frac{\hat\gamma(u+v)+\hat\gamma(u-v)}2
 \end{equation}
 is real analytically connected to the image of
 $\phi$ along $\gamma$ as a time-like minimal immersion.
 Conversely, any real analytic zero mean curvature immersion
 which changes type across a non-degenerate
 null curve is obtained as a real analytic extension
 of non-degenerate fold singularities of a maxface.
\end{theorem}

This assertion follows immediately from 
Fact \ref{thm:KKSY} and the following
Lemmas \ref{lem:fold} and \ref{lem:max-ext}.
\begin{lemma}\label{lem:fold}
 Let $\gamma:(a,b)\to \R^3_1$
 be a real analytic non-degenerate null curve.
 Then
 \[
    \phi (u+iv):=\frac{\gamma(u+iv)+\gamma(u-iv)}2
 \]
 gives a maxface with non-degenerate fold singularities on
 the real axis. 
\end{lemma}

\begin{proof}
 We set $z=u+iv$.
 Then it holds that
 \[
   \phi_z=\frac12(\phi_u-i \phi_v)=\frac{1}{2}\gamma'(u+iv),
 \]
 where $\gamma'(t):=d\gamma(t)/dt$.
 Since $\gamma$ is a regular real analytic curve,
 the map
 \begin{equation}\label{eq:Fdef}
     \Phi(u+iv):=\gamma(u+iv)
 \end{equation}
 gives a null holomorphic immersion if
 $v$ is sufficiently small.
 Thus $\phi=\Re(\Phi)$ gives a maxface.

 Since $\gamma$ is a null curve,
 it holds that
 \begin{equation}\label{eq:sigma}
  \gamma'_0(t)^2=\gamma'_1(t)^2+\gamma'_2(t)^2,
 \end{equation}
 where we set
 $\gamma=(\gamma_0,\gamma_1,\gamma_2)$.
 Moreover, since $\gamma$ is a regular curve,
 \eqref{eq:sigma} implies
 \begin{equation}\label{eq:s0}
  \gamma'_0(t)\ne 0 \qquad (a<t<b).
 \end{equation}
 It can be easily checked that the maxface
 $\phi$ has the Weierstrass data
 \begin{align}\label{eq:w}
  \eta&:=\frac12(d\Phi_1-id\Phi_2)=
          \frac{\gamma'_1(z)-i\gamma'_2(z)}{2}dz,\\
     g&:=-\frac{d\Phi_0}{2\eta}=
          -\frac{\gamma'_0(z)}{\gamma'_1(z)-i\gamma'_2(z)}
         =
          -\frac{\gamma'_1(z)+i\gamma'_2(z)}{\gamma'_0(z)},
  \label{eq:g}
 \end{align}
 where we set $\Phi=(\Phi_0,\Phi_1,\Phi_2)$
 and use the identity
 \[
    (\gamma'_1-i\gamma'_2)(\gamma'_1+i\gamma'_2)
             =(\gamma'_1)^2+(\gamma'_2)^2=
          (\gamma'_0)^2.
 \]
 In particular, \eqref{eq:g} implies that
 $|G|=1$
 holds on the $u$-axis,
 which implies that the $u$-axis consists of singular points.
 By \eqref{eq:g},
 $dG$ vanishes if and only if
 \[
    \Delta:=(\gamma'_1+i\gamma'_2)'\gamma'_0-
         (\gamma'_1+i\gamma'_2)\gamma''_0
      =(\gamma'_0\gamma''_1-\gamma'_1\gamma''_0)
     +i(\gamma'_0\gamma''_2-\gamma'_2\gamma''_0)
 \]
 vanishes. 
 In other words, $\Delta=0$ 
 if and only if
 $(\gamma''_0,\gamma''_j)$
 is proportional to $(\gamma'_0,\gamma'_j)$
 for $j=1,2$, namely
 $\gamma''$ is proportional to $\gamma'$.
 Since $\gamma$ is non-degenerate,
 this is impossible.
 So the image of the curve $\gamma$
 consists of
 non-degenerate singular points
 (cf.\ Definition \ref{def:dg}).

 By \eqref{eq:s0}, \eqref{eq:w} and \eqref{eq:g},
 we have that
 \[
    \frac{dz}{G^2\eta}=\frac{2(\gamma'_1-i\gamma'_2)}{(\gamma'_0)^2}.
 \]
 Thus the $u$-axis consists of non-degenerate 
 fold singular points if and only if
 $\Delta_1:=(\gamma'_1-i\gamma'_2)\Delta$
 is a real valued function.
 Here
 \[
    \Re(\Delta_1)=
     \gamma'_0(\gamma'_1\gamma''_1+\gamma'_2\gamma''_2)-
       \gamma''_0((\gamma'_1)^2+(\gamma'_2)^2)
       =\gamma'_0(\gamma'_0\gamma''_0)
        -\gamma''_0(\gamma'_0)^2=0,
 \]
 where we used the identity
 \eqref{eq:sigma} and its derivative.
 This implies that $\gamma$
 consists of non-degenerate fold singularities.
\end{proof}

Finally, we prove the converse assertion:

\begin{lemma}\label{lem:max-ext}
 Let $\phi:\Sigma^2\to \R^3_1$ be a maxface
 which has non-degenerate fold singularities along
 a singular curve $\gamma: (a,b) \rightarrow \Sigma^2$.
 Then, the space curve
 $\hat \gamma(t):=\phi\circ \gamma(t)$
 is a non-degenerate real analytic null curve
 such that
 \[
    \hat \phi(u,v):=\frac12\biggl(
          \hat\gamma(u+iv)+\hat\gamma(u-iv)\biggr)
 \]
 coincides with the original maxface $\phi$.
 In particular, $\hat{\phi}$ satisfies the identity
 $\hat{\phi}(u,v)=\hat{\phi}(u,-v)$.
\end{lemma}
\begin{proof}
 The singular set of $\phi$ can be characterized 
 by the set $|G|=1$,
 where $(G,\eta)$ is the Weierstrass data as in \eqref{eq:Wei}.
 Let $T$ be a M\"obius transformation
 on $S^2=\C\cup \{\infty\}$ which
 maps the unit circle $\{\zeta\in\C\,;\,|\zeta|=1\}$ to the real axis.
 Then $T\circ G$ maps the image of the singular curve $\gamma$ 
 to the real axis.
 Since $dG\ne 0$ (cf. Definition \ref{def:fold}), we can 
 choose $T\circ G$ as a local complex coordinate.
 We denote it by
 \[
    z=u+iv.
 \]
 Then  the image of $\gamma$ coincides with the real axis $\{v=0\}$.
 Let $\Phi$ be the holomorphic lift of $\phi$.
 Since the real axis consists of non-degenerate fold singularities,
 Proposition~\ref{prop:dual} implies that
 $\Im(\Phi)$ is constant on the real axis.
 Since $\Phi$ has an ambiguity of translations by
 pure imaginary vectors,
 we may assume without loss of generality that
 \begin{equation}\label{eq:im-zero}
   \Im (\Phi)=0\qquad\text{on the real axis}.
 \end{equation}
 Thus,
 the curve $\hat \gamma$
 is expressed by (cf.\ \eqref{eq:Fdef})
 \begin{equation}\label{eq:hat-gamma}
   \hat\gamma(u) = \phi(u,0)=\Re\bigl(\Phi(u,0)\bigr)
                 = \Phi(u,0),
 \end{equation}
 namely, the two holomorphic functions
 $\Phi(u+iv)$ and $\hat\gamma(u+iv)$
 take the same values on the real axis.
 Hence
 $\Phi(z)=\hat\gamma(u+iv)$ and thus
 \[
   \phi(z)=\frac{\hat\gamma(u+iv)+\hat\gamma(u-iv)}2=\hat\phi(u,v).
 \]
 So it is sufficient to show that
 $\hat\gamma(t)$ is a non-degenerate 
 null curve.
 Since $|G|=1$ on the real axis,
 there exists a real-valued function $t=t(u)$ such that
 \begin{equation}\label{eq:g-unit}
  G(u)=e^{it(u)} \qquad (u\in \R).
 \end{equation}
 Differentiating this along the real axis, we have
 $G_u(u)=i e^{it}({dt}/{du})$.
 Here, $dt/du$ does not vanish because $dG\neq 0$ on the real axis.
 Since $\gamma$ consists of non-degenerate fold singularities,
 \[
   i\frac{dG}{G^2\eta}=-\frac{e^{-it(u)}}{w(u)}\frac{dt}{du}
 \]
 must be real valued (cf.\ Definition~\ref{def:fold}), 
 where $\eta=w(z)\,dz$.
 Since
 $\Phi$ is an immersion, 
 $G$ must have  a pole
 at $z=u$ if 
 $w(u)=0$ (cf.\ \eqref{eq:Wei}), 
 but this contradicts the fact that $|G|=1$ along $\gamma$.
 Thus we have $w(u)\ne 0$.
 It then follows that
 \[
    \xi(u):=e^{it(u)}w(u)=G(u)w(u)
 \]
 is a non-vanishing real valued analytic function.
 Now, if we write
 $\hat\gamma=(\hat\gamma_0,\hat\gamma_1,\hat\gamma_2)$, 
 \eqref{eq:Wei} yields
 \[
   \hat \gamma'_0(u)=\Re\biggl(-2 G(u) w(u)\biggr)=-2 \xi(u),
 \]
 and
 \begin{align*}
  \hat \gamma'_1(u)&=\Re\biggl((1+G(u)^2)w(u)\biggr)=2 
                  \xi(u)\cos t(u),\\
  \hat \gamma'_2(u)&=\Re\biggl(i(1-G(u)^2)w(u)\biggr)=2  
                  \xi(u)\sin t(u).
 \end{align*}
 This implies that
 $\hat \gamma(u)$ is a regular real analytic null curve.
 Since $dt/du\ne 0$, the acceleration vector
 \[
   \hat \gamma''(u)=\bigl(\log \xi(u)\bigr)'\gamma'(u)+2\xi(u)
           \biggl(0,-\sin t(u),\cos t(u)\biggr)
           \frac{dt}{du}
 \]
 is not proportional to $\hat \gamma'(u)$.
 Thus, $\hat \gamma(u)$ is non-degenerate.
\end{proof}

\begin{corollary}\label{cor:main}
 Let $\phi:\Sigma^2\to \R^3_1$ be a maxface.
 Then a singular point $p\in \Sigma^2$
 lies on a non-degenerate fold singularity if and only 
 if there exists a local complex coordinate
 $z=u+iv$ with $p=(0,0)$ satisfying the
 following two properties:
 \begin{enumerate}
  \item\label{item:main:1} 
       $\phi(u,v)=\phi(u,-v)$,
  \item\label{item:main:2}
       $\phi_{u}(0,0)$ and $\phi_{uu}(0,0)$
	are linearly independent.  
 \end{enumerate}
\end{corollary}
\begin{proof}
 Suppose $p$ is a non-degenerate fold singularity of $\phi$.
 Then \ref{item:main:1} follows from Lemma \ref{lem:max-ext},
 and \ref{item:main:2} follows from the fact that 
 the null curve parameterizing fold
 singularities is non-degenerate.

Conversely, suppose there is a coordinate system 
around a singular point 
$p$ with $p=(0,0)$ which satisfies 
\ref{item:main:1} and \ref{item:main:2}.
Differentiating \ref{item:main:1}, we have that
$\phi_v(u,0)=0$. 
Let $\Phi$ be a holomorphic lift of $\phi$.
Since $\Phi$ is a null holomorphic map,
the relation
$\phi_v(u,0)=0$ implies that
\[
   0=\inner{\Phi_z(u,0)}{\Phi_z(u,0)}
    =4\inner{\phi_z(u,0)}{\phi_z(u,0)}
    =\inner{\phi_u(u,0)}{\phi_u(u,0)},
 \]
where $\inner{}{}$ is the canonical inner product
of $\R^3_1$. 
This implies $\gamma:u\mapsto \phi(u,0)$ is a null curve.
By the condition (2), this null curve is non-degenerate.  
Then by Lemma \ref{lem:fold},
 \[
    \phi_\gamma(u+iv):=\frac{\gamma(u+iv)+\gamma(u-iv)}2
 \]
 is a maxface such that $\gamma$  parametrizes a non-degenerate fold
 singularity of $\phi_\gamma$.
 Moreover,
 $\Phi_\gamma:=\gamma(u+iv)$
 gives a holomorphic lift of $\phi_\gamma$
 (cf.\ \eqref{eq:Fdef}).
 Since
 \[
    \Phi_z(u,0)=2\phi_z(u,0)=\gamma'(u)
               =(\Phi_\gamma)_z(u,0),
 \]
 the holomorphicity of $\Phi$ and $\Phi_\gamma$
 yields that $\Phi_z(z)$ coincides with $(\Phi_\gamma)_z(z)$.
 Thus $\Phi$ coincides with $\Phi_\gamma$ up to a constant.
 Then $\phi_\gamma$ coincides with $\phi$, 
 and we can conclude that $\gamma$ parametrizes a
 non-degenerate fold singularity of $\phi$.
\end{proof}

So far, we have looked at the singular curve of maxfaces. 
Now we turn our attention to 
the singular curve of zero mean curvature surfaces and 
prove the following assertion, 
which can be considered
as the fundamental theorem of type change 
for zero mean curvature surfaces:

\begin{theorem}[Gu \cite{G,G1, G2} and Klyachin \cite{Kl}]
\label{thm:main2}
 Let $\gamma:(a,b)\to \R^3_1$ 
 be a non-degenerate real analytic null curve. 
 We set
 \[
    \hat\phi_{\gamma}(u,v):=
      \begin{cases}
       \displaystyle\frac{\gamma(u+i\sqrt{v})+\gamma(u-i\sqrt{v})}2 
               & (v\ge 0),\\[10pt]
       \displaystyle\frac{\gamma(u+\sqrt{|v|})+\gamma(u-\sqrt{|v|})}2 
               & (v<0),
      \end{cases}
 \]
 for sufficiently small $|v|$.
 Then $\hat\phi_{\gamma}$ gives a real analytic 
 zero mean curvature immersion 
 such that the image of 
 $\gamma$ consists of non-degenerate points 
 of type change with respect to $\hat \phi_\gamma$.

 Conversely, let $f:\Omega^2\to \R$ 
 be a $C^\infty$-function satisfying the
 zero mean curvature equation \eqref{zm},
 and let $p=(x_0,y_0)$ be a  non-degenerate point 
 of type change with respect to $f$, where $\Omega^2$
 is a domain in the $xy$-plane. Then there exists a real analytic 
 non-degenerate null
 curve $\gamma$ in $\R^3_1$ through
 $(f(x_0,y_0),x_0,y_0)$
 with $\hat \phi_\gamma$ coinciding with the graph of $f$ in 
 a small neighborhood of $p$.
\end{theorem}

This assertion was proved by Gu \cite{G,G1,G2}.
Later, Klyachin \cite{Kl} analyzed type-changes
of zero mean curvature surfaces not only
at non-degenerate points of type changes
but also degenerate cases as mentioned in
the introduction, and got the same assertion
as a corollary.  
Note that the conclusion for regularity of the converse 
statement is stronger than  that of 
Proposition \ref{prop:Kl2}.
We remark that Gu \cite{G2}
gave a generalization of Theorem \ref{thm:main2}
for $2$-dimensional zero mean curvature surfaces in $\R^{n+1}_1$
 ($n\ge 2$).
\begin{proof}
 We have already proved the first assertion.
 (In fact, Proposition \ref{prop:Kl1}
 implies that $\gamma$ consists of non-degenerate points 
 of type change with respect to $\hat\phi_\gamma$.)
 Then it is sufficient to show the converse assertion,
 which is proved in \cite{G,G1,G2} and  \cite{Kl}.
 Here referring to \cite{Kl},
 we give only a sketch of the proof:
 Let $f:\Omega^2\to \R$ be a 
 $C^\infty$-function
 satisfying the zero mean curvature equation \eqref{zm},
 and let $p=(x_0,y_0)\in \Omega^2$ be a non-degenerate point 
 of type change with respect to $f$.
 (As pointed out in Gu \cite{G1,G2} and Klyachin \cite{Kl},
 one can prove the real analyticity of $f$ at $p$,
 assuming only $C^3$-regularity of $f$, using the same
 argument as below.)
 By Proposition \ref{prop:Kl2},
 there exists a $C^\infty$-regular 
 curve $\sigma(u)$ ($|u|<\delta$) such that 
 \[
    \gamma(u):=(f\circ \sigma(u),\sigma(u))
 \]
 is a non-degenerate null curve passing through
 $(f(x_0,y_0),x_0,y_0)$,
 where $\delta$ is a sufficiently small 
 positive number.
 We set $B:=1-f_x^2-f^2_y$.
 Let $\Omega^+$ be a simply connected
 domain such that $B>0$, 
 and  suppose that $\sigma$
 lies on the boundary of $\Omega^+$.
 We set
 \[
     t = t(x,y) := f(x,y),\qquad
     s = s(x,y):= \int_{q_0}^q \frac{-f_ydx+f_x dy}{\sqrt{B}},
 \]
 where $q_0\in \Omega^+$ is a base point and $q:=(x,y)$.
 Since
 $\alpha:=(-f_ydx+f_x dy)/B$ is a closed $1$-form,
 its (line) integral $\displaystyle\int_{q_0}^q \alpha$
 does not depend on the choice of path.
 Let $\tau(v)=(a(v),b(v))$ ($0\le v\le \epsilon$)
 be a path starting from $p$ and going into $\Omega^+$
 which is transversal to the curve $\sigma$.
 Since $B(p)=0$ and $\nabla B(p)\ne 0$,
 there exists a constant $C>0$ such that
 $B\circ \tau(v)=C v + \mathcal{O}(v^2)$,
 where $\mathcal{O}(v^2)$ denotes  the higher order terms.  
 Then there exists a constant $m$ such that
 \[
    \left| \frac{-f_y\circ \tau(v)\,  \dfrac{da}{dv}(v) 
                + f_x\circ \tau(v)  \,\dfrac{db}{dv}(v)}
          { \sqrt{ B\circ \tau(v)}}  \right| < \frac{m}{\sqrt{v}} 
         \quad\text{ for } 0 < v\le \epsilon,
 \]
 hence
 \[
       \lambda:=\int_\tau |\alpha|<  \int_0^\epsilon 
              \frac{m}{\sqrt{v}}dv<\infty,
 \]
 which is just the case (1) of \cite[Lemma 6]{Kl},
 and $(t,s)$ gives an isothermal coordinate of $\Omega^+$
 with respect to the immersion 
 \[
      \phi:(t,s)\mapsto
    (f(x(t,s),y(t,s)),x(t,s),y(t,s))=(t,x(t,s),y(t,s)) \qquad (s>0).
 \]
 Moreover, the function $s(x,y)$ can be continuously
 extended to the image of the curve $\sigma$.
 Since $\sigma$ is an integral curve of $\nabla f$,
 we may assume that $\sigma$ parametrizes the
 level set $s=0$, where we have
 used the fact that $s(x,y)$ is constant along each
 integral curve of $\nabla f$.
 In particular, $\phi$ satisfies $\phi_{tt}+\phi_{ss}=0$.
 Then $\phi(t,s)$ can be extended to a 
 harmonic $\R^3$-valued function
 for $s<0$ satisfying $\phi(t,s)=\phi(t,-s)$
 via the symmetry principle (see the proof of \cite[Theorem 6]{Kl}).
 In particular, $f$ 
 is a real analytic function
 whose graph coincides with the image of $\phi$ on $\Omega^+$
 near $p$. 
 Moreover $t\mapsto \phi(t,0)$ parametrizes the curve
 $\gamma$ (cf.\ \cite[Page 219]{Kl}).
 By Corollary \eqref{cor:main},
 $\gamma$ can be considered as a non-degenerate fold singularity
 of the maxface $(t,s)\mapsto \phi(t,s)$.
 (In fact, the condition \ref{item:main:2} of Corollary \eqref{cor:main}
 corresponds to the fact that $\gamma$ is a non-degenerate
 curve near $p$.)
 Then Theorem \ref{thm:main} implies
 that $\hat\phi_\gamma$ coincides with the graph
 of $f$ on a sufficiently small neighborhood of $p$.
\end{proof}

As an application of Theorem \ref{thm:main2},
embedded triply periodic zero mean curvature surfaces of mixed type 
in $\R^3_1$ with the same topology 
as the Schwarz D surface in the Euclidean 3-space $\R^3$
have been constructed, in \cite{FRUYY2}.

\section{%
 The conjugates of hyperbolic catenoids and Scherk-type surfaces
}\label{sec:conjugate}

The two entire graphs of 
$n$ variables
\begin{align*}
 f_1(x_1,\dots,x_n):=x_1\tanh(x_2),\qquad
 f_2(x_1,\dots,x_n):=(\log \cosh x_1)-(\log \cosh x_2),
\end{align*}
given by Osamu Kobayashi \cite{K},
are zero mean curvature hypersurfaces in $\R^{n+1}_1$
which change type from space-like to time-like.
When $n=2$, the image of $f_1$ is congruent to $\Cat_0$
and the image of $f_2$ is congruent to $\Sch_0$.
On the other hand,
the space-like catenoid $\Cat_+$
(resp.\ the space-like Scherk surface $\Sch_+$)
and the time-like catenoid $\Cat_-$ 
(resp.\ the time-like Scherk surface $\Sch_-$)
are typical examples of zero mean curvature surfaces 
which contain singular light-like lines.
Moreover, they are closely related
to $\Cat_0$ (resp.\ $\Sch_0$)
by taking their conjugate surfaces
as follows:

\begin{fact}[\cite{K}, \cite{KKSY} and \cite{CR}]
 The conjugate space-like maximal surface of
 the space-like hyperbolic catenoid $\Cat_+$
 and
 the conjugate time-like minimal surface of
 the time-like hyperbolic catenoid $\Cat_-$
 are both congruent to subsets of the entire graph $\Cat_0$.
\end{fact}

The space-like hyperbolic catenoid $\Cat_+$ was originally given
in Kobayashi \cite{K} as the catenoid of 2nd kind, and  
Kobayashi also pointed out that 
the space-like part $\Cat_0^+$ of $\Cat_0$ is the conjugate surface of
$\Cat_+$. $\Cat_0^+$ is connected and is called
the {\it space-like hyperbolic helicoid}. 
The time-like part $\Cat_0^-$ of $\Cat_0$ splits into 
two connected components, each of 
which is congruent to
the {\it time-like hyperbolic helicoid\/}
(see Figure \ref{fig1c}, left).
The entire assertion, including the
case of the time-like part,
has been pointed out
in \cite[Lemma 2.11 (3) ]{KKSY} and 
 the caption of Figure 1 in \cite{CR} without proof.

\begin{proof}
 We give a proof here as an application of the results in
 the previous section.
 A subset of the space-like hyperbolic catenoid $\Cat_+$
 can be parametrized by 
 \begin{align}\label{eq:s-cat-par}
   \varphi_1(u,v)&=(\cosh u\sin v, v, \sinh u \sin v)
                 =-\Re i(\sinh z, z, \cosh z),\\
   \psi_1(u,v)&=(-\cosh u\sin v, v, -\sinh u \sin v),\nonumber
 \end{align}
 where $z=u+iv$. 
 In fact, $\Cat_+$
 is the union of the closure of the images of $\varphi_1$ and $\psi_1$.
 The surface $\varphi_1$ has generalized conical singularities 
 at $(u,n \pi)$ for any $u \in \R$ and $n \in \Z$,
 as pointed out in  \cite{CR}.
 Using this, one can easily compute that
 the conjugate of $\varphi_1$ is 
 congruent to the following surface
 \begin{align}\label{eq:s-cat-dual-par}
    \varphi_1^*(u,v):&= -\Im i(\sinh z, z, \cosh z)
                = -\Re(\sinh z , z , \cosh z)\\
                &= -(\sinh u\cos v, u,\cosh u\cos v).\nonumber
 \end{align}
 By Proposition \ref{prop:dual},
 the conjugate surface has non-degenerate fold singularities.
 Then by Theorem \ref{thm:main}
 one can get an analytic continuation of
 $\varphi_1^*$ as a zero mean curvature surface in
 $\R^3_1$ which changes type across the  fold singularities.
 We can get an explicit description of
 such an extension of $\phi_1^*$ as follows:
 We set
 \[
    (t,x,y)=\varphi_1^*(u,v)=-(\sinh u\cos v, u,\cosh u\cos v).
 \]
 Then the surface has fold singularities 
 at $(u,n \pi)$ for any $u \in \R$ and $n \in \Z$.
 Then it holds that 
 \[
    \frac{t}{y}=\tanh u = -\tanh x
 \]
 and the image of $-\varphi_1^*$ is contained in the surface $\Cat_0$.

 On the other hand, a subset of the time-like 
 hyperbolic catenoid $\Cat_-$ has a parametrization 
 \begin{equation}\label{eq:t-cat-par}
    \varphi_2(u,v):=\frac{1}{2}(\sinh u+\sinh v,u+v,\cosh u-\cosh v) 
         = \frac{\alpha(u)+\beta(v)}{2},
 \end{equation}
 where
 \begin{equation}\label{eq:alpha-beta}
   \alpha(u):=(\sinh u, u, \cosh u), \qquad
   \beta(v):= (\sinh v, v, -\cosh v).
 \end{equation}
 Also
 \begin{equation}\label{eq:t-cat-par2}
    \psi_2(u,v):=\frac{1}{2}(-\sinh u-\sinh v,u+v,-\cosh u+\cosh v) 
 \end{equation}
 gives a parametrization of $\Cat_-$.
 More precisely, $\Cat_-$ is 
 the union of the closure of the images of $\varphi_2$ and 
 $\psi_2$. We get
 the following parametrization of
 the conjugate surface $\varphi_2^*$ of $\varphi_2$
 \begin{equation}\label{eq:t-cat-dual-par}
  \varphi_2^*(u,v):= \frac{1}{2}\bigl(\alpha(u)-\beta(v)\bigr)
        = \frac{1}{2}(\sinh u-\sinh v, u-v,\cosh u+\cosh v),
 \end{equation}
 where $\alpha$ and $\beta$ are as in \eqref{eq:alpha-beta}.
 (See \cite{IT} for the definition
 of the conjugate surfaces of time-like minimal surfaces.)
 To find the implicit function representation of the image of
 $\varphi_2^*$,
 take a new coordinate system $(\xi,\zeta)$ as
 \[
    u=\xi+\zeta, \qquad v=\xi-\zeta.
 \]
 Then
 \[
    \varphi_2^*(\xi,\zeta)=
      (\cosh\xi\sinh\zeta,\zeta,\cosh\xi\cosh\zeta),
 \]
 which implies that
 the image is a subset of $\Cat_0$.
 The entire graph $\Cat_0$ changes type across 
 two disjoint real analytic  null curves $\{y=\pm\cosh x\}$.
\end{proof}

Next, we prove a similar assertion for
the Scherk surfaces, which is also briefly mentioned
in the caption of Figure 1 in \cite{CR}:

\begin{theorem}
 The conjugate space-like maximal surface of
 the space-like Scherk surface $\Sch_+$
 and the 
 conjugate time-like minimal surface of
 the time-like Scherk surface $\Sch_-$
 of the first kind
 are both congruent to subsets
 of the entire graph $\Sch_0$.
\end{theorem}

\begin{proof}
 Using the identities
 \[
  \cos \arg z=\Re(\frac{z}{|z|}),\qquad \sin \arg z=\Im(\frac{z}{|z|}),
 \]
 one can prove that
 a subset of the space-like Scherk surface $\Sch_+$ 
 is parametrized by the complex variable $z$ as 
 \begin{align}\label{eq:sp-sch-par}
  \varphi_1(z)&=-\Re i\left(
		\log\frac{1+z^2}{1-z^2},\,\,
		\log\frac{1-z}{1+z},\,\,
		\log\frac{1-iz}{1+iz}
	       \right)+\frac{\pi}{2}(1,1,1)\\
          &=\left(\arg\frac{1+z^2}{1-z^2},\,\,
                  \arg\frac{1-z}{1+z},\,\,
                  \arg\frac{1-iz}{1+iz}\right)
           +\frac{\pi}{2}(1,1,1).\nonumber
 \end{align}
 In fact, $\varphi_1(z)$ is a multi-valued $\R^3_1$-valued function,
 but can be considered as a single-valued function
 on the universal covering of $\C\cup\{\infty\}\setminus \{\pm 1,\pm i\}$.
 We now set
 \[
   \varphi_1(z)=\bigl(t(z),x(z),y(z)\bigr)
 \]
 and
 \[
   \psi_1(z):=\bigl(\pi-t(z),x(z),\pi-y(z)\bigr).
 \]
 Then $\Sch_+$
 is the union of the closure of the images of $\varphi_1$ and 
 $\psi_1$.
 The conjugate $\psi_1^*$ of the space-like Scherk 
 surface $\psi_1$  is obtained by
 \begin{align}\label{eq:sp-sch-dual-par}
  \psi_1^*(z)&=
             \Im i\left(
		\log\frac{1+z^2}{1-z^2},\,\,
		\log\frac{1-z}{1+z},\,\,
		\log\frac{1-iz}{1+iz}
	       \right)\\
          &=\left(\log\left|\frac{1+z^2}{1-z^2}\right|,\,\,
                  \log\left|\frac{1-z}{1+z}\right|,\,\,
                  \log\left|\frac{1-iz}{1+iz}\right|
            \right)
            .\nonumber
 \end{align}
 Since $\psi_1$ admits only generalized cone-like singularities
 (cf.\ Proposition \ref{prop:dual}),
 $\psi_1^*$ admits only fold singularities, and has
 a real analytical extension across  the 
 fold singularities to a
 time-like minimal surface in $\R^3_1$ (cf.\ Theorem \ref{thm:main}).
 More precisely,
 the image of the conjugate $\psi_1^*$ is contained in the graph
 $\Sch_0$, shown as follows:
 The singular sets of $\psi_1$ and $\psi_1^*$ are both parametrized
 as $\{z=e^{iu}\}$.
 Then the image of a connected component of singular curve by $\psi_1^*$
 as in 
 \eqref{eq:sp-sch-dual-par} is parametrized as 
 \begin{equation}\label{eq:null-curve}
  \gamma(u) = \frac{1}{2}\left(2\log\cot u,\,\,
                  \log \frac{1-\cos u}{1+\cos u},\,\,
		  \log \frac{1+\sin u}{1-\sin u}
			 \right)
  \qquad \left(0<u<\frac{\pi}2\right).
 \end{equation}
 By the singular Bj\"orling formula \eqref{eq:Bj}
 in Section~\ref{sec:fundamental}, 
 we have the following analytic extension of $\gamma$
 \begin{equation}\label{eq:t-sch-dual-par}
  \hat\psi_2(u,v):=\frac{\gamma(u)+\gamma(v)}{2}.
 \end{equation}
 Now, we check that the conjugate 
 $\psi_2:=\hat\psi_2^*$ of 
 $\hat \psi_2$ as in  \eqref{eq:t-sch-dual-par} coincides with 
 the time-like Scherk surface $\Sch_-$.
 By \eqref{eq:t-sch-dual-par}, the conjugate $\psi_2$ of $\hat \psi_2$ is
 parametrized by
 (see \cite{IT} for the definition
 of the conjugate surfaces of time-like minimal surfaces)
 \begin{align}\label{eq:t-sch-par}
   \psi_2&(u,v)=\frac{1}{2}\bigl(\gamma(u)-\gamma(v)\bigr)\\
     &= \left(
         \frac{1}{2}\bigl(
	 \log(\cot u)-\log(\cot v)\bigr),
	 \frac{1}{4}
         \left(\log\frac{1-\cos u}{1+\cos u}-
       \log\frac{1-\cos v}{1+\cos v}\right),\right.\nonumber\\
     & \left.\hphantom{=\left(
         \frac{1}{2}\bigl(
	 \log(\cot u)+\log(\cot v)\bigr)\right.,}
	 \frac{1}{4}
         \left(\log\frac{1+\sin u}{1-\sin u}-
              \log\frac{1+\sin v}{1-\sin v}\right)
       \right)\nonumber.
 \end{align}
 We set $(t,x,y)=\psi_2(u,v)$, and will show that
 $(t,x,y)$ lies in $\Sch_-$:
 In fact, by \eqref{eq:t-sch-par}, we have
 \[
  e^{2t}=\frac{\cot u}{\cot v}=\frac{\cos u\sin v}{\sin u \cos v},
 \]
 which implies
 \[
 \cosh t =
         \frac{1}{2}\left(
         \sqrt{\frac{\cos u\sin v}{\sin u \cos v}}+
         \sqrt{\frac{\sin u\cos v}{\cos u \sin v}}\right)
    = \frac{\sin(u+v)}{\sqrt{\sin 2 u \sin 2 v}}.
 \]
 Using
 \[
    e^{4x}=\frac{1-\cos u}{1+\cos u}\times 
           \frac{1+\cos v}{1-\cos v}
           =\left(\tan\frac{u}{2}\cot\frac{v}{2}\right)^2,
 \]
 we have that
 \[
   \cosh x =
        \frac{1}{\sqrt{\sin u\sin v}}\sin\frac{u+v}{2}.
 \]
 Similarly, 
 \[
     \cosh y=
           \frac{1}{\sqrt{\cos u\cos v}}\cos\frac{u+v}{2}
 \]
 holds. 
 Hence the analytic extension of the conjugate of 
 $\Sch_+$ coincides with $\Sch_0$.
 As pointed out in \cite{KKSY},
 the space-like part of $\Sch_0$ is connected,
 and the time-like part of $\Sch_0$
 consists of four connected components, 
 each of which is congruent to the image of $\psi_1$
 (see Figure~\ref{fig1c}, right). 
\end{proof}

\medskip
In the introduction, we saw that the conjugate
surfaces of $\Cat_+$ and $\Cat_-$
(resp.\ $\Sch_+$ and $\Sch_-$) 
are both subsets of the same zero mean curvature surface 
$\Cat_0$ (resp.\ $\Sch_0$).
As was pointed out in \cite{KKSY}, 
a similar phenomenon also holds 
between elliptic catenoids and parabolic catenoids:
The helicoid 
$x\sin t=y \cos t$
is well known as a ruled minimal surface in 
the Euclidean 3-space $\R^3$, and also gives a zero-mean curvature
in $\R^3_1$.
The {\it space-like elliptic catenoid\/}
\[
   \phi^E_+(u,v):=
     \left(v, \cos u \sinh v,\sin u \sinh v \right)
\]
and the {\it time-like elliptic catenoid\/}
\[
   \phi^E_-(u,v):=
        \left(v, \cosh u \sinh v,\sinh u \sinh v
        \right)
\]
induce their conjugate surfaces, both of which 
are subsets of the helicoids.
\begin{figure}[t]%
 \begin{center}
  \begin{tabular}{c@{\hspace{5em}}c}
       \includegraphics[width=4.2cm]{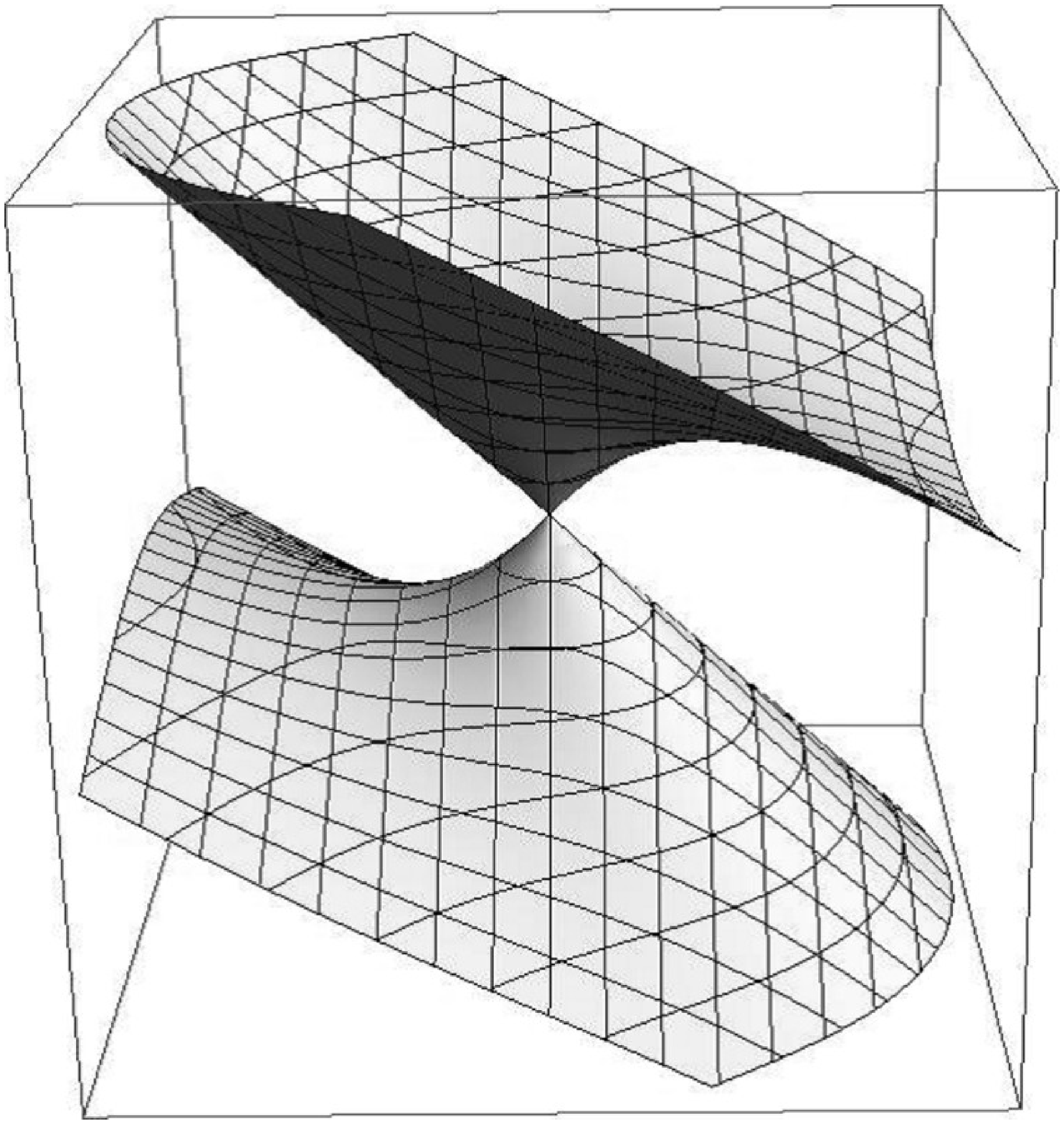} & 
       \includegraphics[width=4.5cm]{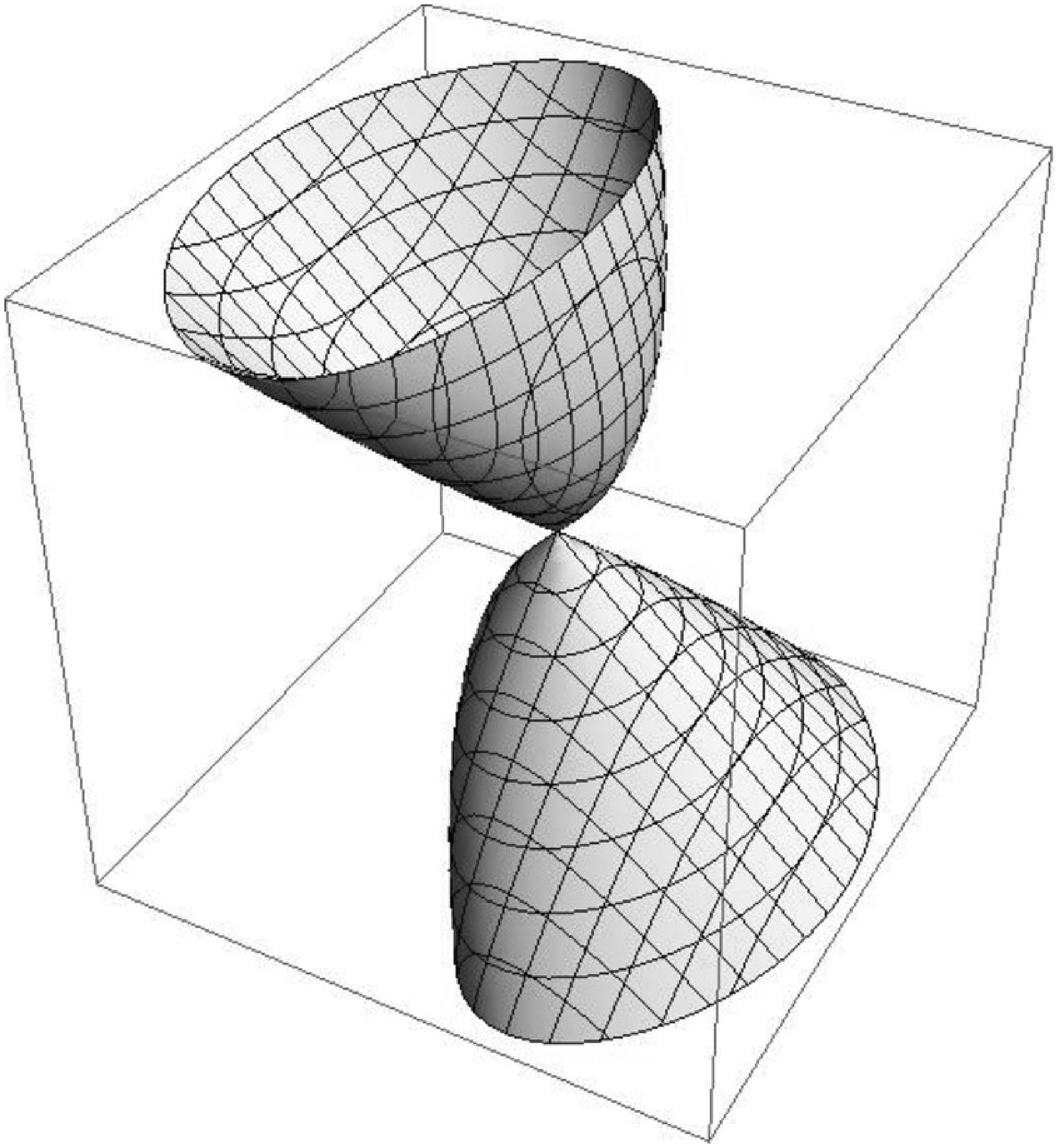} 
  \end{tabular}
 \end{center}
 \caption{%
  The completions of parabolic catenoids 
  $\phi^P_+$ (left) and $\phi^P_-$ (right).
 }%
\label{fig:pc}
\end{figure}

The {\it space-like parabolic catenoid},
on the other hand,
\[
   \phi^P_+(u,v):=
       \left(v-\frac{v^3}3+u^2v, v+\frac{v^3}3-u^2v,2uv
       \right)
\]
is given in Kobayashi \cite{K} as an Enneper surface of 2nd kind.
Consider the ruled zero-mean curvature
surface, which we call the {\it parabolic helicoid}
\[
    \phi^P_0(u,v):=\gamma(u)+v(u,-u,1),\qquad
    \left(
      \gamma(u):=\left(-u-\frac{u^3}3,-u+\frac{u^3}3,-u^2\right)
    \right),
\]
where $\gamma(u)$ is a non-degenerate null curve
at which the surface changes type. 
If $v>0$, $\phi^P_0$ gives the conjugate surface
of $\phi^P_+$.
If $v<0$, $\phi^P_0$ gives the conjugate surface
of the {\it time-like parabolic catenoid}%
\footnote{%
   In \cite[Examples 2.8 and 2.9]{KKSY}, these 
   are called spacelike (timelike) parabolic helicoid.
}%
given by
\[
  \phi^P_-(u,v):=
      \left(-u-\frac{u^3}3-uv^2,-u+\frac{u^3}3+uv^2,-u^2-v^2
      \right).
\]
The image of the parabolic catenoid
$\phi^P_+(u,v)$ (resp. $\phi^P_-(u,v)$)
is a subset of (cf. Figure \ref{fig:pc})
\begin{equation}\label{eq:p2}
 12 (x^2 - t^2) = (x + t)^4 - 12 y^2
  \qquad \biggl (
  \mbox{resp.\ }\,\, 12 (x^2 - t^2) = -(x + t)^4 - 12 y^2
  \biggr).
\end{equation}
We call it {\it the completion of the space-like} 
({\it resp.\ time-like}) {\it parabolic catenoid},
which contains a light-like line
\[
   L:=\{y=x+t=0\}.
\]
The initial parametrization $\phi^P_\pm$ does not
include this line $L$.
Kobayashi \cite[Example 2.3]{K} noticed the line $L$ 
in the surface and drew a 
hand-drawn figure of it 
that coincides with the left-hand 
side of Figure \ref{fig:pc}.

In \cite{CR}, zero mean curvature
surfaces containing singular light-like lines
are categorized into the following six classes 
\begin{equation}\label{eq:six}
 \alpha^+,\quad\alpha^0_I,\quad
  \alpha^0_{\II},\quad
  \alpha^-_I,
  \quad\alpha^-_{\II},\quad
  \alpha^-_{\III}.
\end{equation}
The surfaces belonging to $\alpha^+$ 
(resp.\ $\alpha^-_I$, $\alpha^-_{\II}$, $\alpha^-_{\III}$)
are space-like (resp.\ time-like).
On the other hand, the causalities of surfaces in
$\alpha^0_I$ and $\alpha^0_{\II}$  are not unique. 
In fact, the hyperbolic catenoids $\Cat_{\pm}$ are the
example of surfaces of type $\alpha^0_I$.
In \cite{CR}, the authors pointed out that the light cone
$\Lambda^2$
is a surface of type $\alpha^0_{\II}$. 
Since $\Lambda^2$ is neither 
space-like nor time-like, the existence problem 
of space-like (resp.\ time-like)
surfaces for this class $\alpha^0_{\II}$
was left unanswered.  Fortunately, the above two examples 
resolve this question, that is, the following assertion holds:

\begin{theorem}
 The completion of space-like {\rm(}resp.\ time-like{\rm)}
 parabolic catenoids gives an example of
 surfaces of type $\alpha^0_{\II}$ at each point 
 $(-c,c,0)$ $(c\ne 0)$
 on the light-like line $L$. 
 In other words, both space-like surfaces
 and time-like surfaces exist
 in the class $\alpha^0_{\II}$
 of zero mean curvature surfaces containing 
 light-like lines. 
\end{theorem}
\begin{proof}
 We set
 \[
    F_\pm:=(x+t)\{12(x-t)\mp (x+t)^3\}+12 y^2.
 \]
 Then \eqref{eq:p2} is rewritten as $F_\pm=0$.
 We fix a point $(-c,c,0)$ ($c\ne 0$)
 on the set $\{F_\pm=0\}$.
 Since
 \[
    \frac{\partial F}{\partial t}(-c,c,0)
       =\left. -24 t\mp 4 (x+t)^3\right|_{(t,x)=(-c,c)}=24c(\ne 0),
 \]
 the implicit function theorem yields that
 there exists a $C^\infty$-function $t=t_\pm(x,y)$
 such that the set $\{F_\pm=0\}$ is parametrized
 by the graph $t=t_\pm(x,y)$ around the point
 $(-c,c,0)$. By \cite{CR}, we know that
 $t=t_\pm(x,y)$ has the following expression
 \[
    t_\pm(c,y)=-c-\frac{\alpha(c)}2 y^2+\beta(c,y)y^3,
 \]
 where $\alpha=\alpha(x)$ and $\beta=\beta(x,y)$
 are $C^\infty$-functions.
 Differentiating the equation
 $F_{\pm}(t_\pm(x,y),x,y)=0$
 with respect to $y$, 
 we have that
 \begin{equation}\label{eq:y}
   -24 t(x,y) t_y(x,y)-4 (x+t(x,y))^3 t_y(x,y)+24 y=0,
 \end{equation}
 where $t_y=\partial t/\partial y$.
 Substituting $(t,x,y)=(-c,c,0)$, we get
 \[
    -ct_y(c,0)=t(c,0)t_y(c,0)=0,
 \]
 which implies that 
 \begin{equation}\label{eq:y2}
    t_y(c,0)=0.
 \end{equation}
 Differentiating \eqref{eq:y}
 with respect to $y$ again, we have
 \[
   -24 t_y^2-24 t \, t_{yy}\mp
     12 (x+t)^2 t_y^2\mp 4(x+t)^3 t_{yy}+24=0.
 \]
 Substituting $(t,x,y)=(-c,c,0)$ 
 and \eqref{eq:y2}, we get
 \[
    -24 c t_{yy}+24=0,
 \]
 namely $\alpha(c)=2/c$. This implies that
 $t=t_\pm(x,y)$ is of type $\alpha^0_{\II}$
 at $(c,0)$.
\end{proof}

In the authors' previous work \cite{Null}, 
surfaces of type $\alpha^0_I$
changing type across a light-like line
have been constructed.
The only other possibility for
the existence of 
surfaces 
changing type across a light-like line
must be of type $\alpha^0_{\II}$
(cf.\ \cite{CR}).
So the following question is of interest:

\begin{problem}
 Do there exist zero-mean curvature surfaces of type $\alpha^0_{\II}$
 which change type across a light-like line?
\end{problem}

\section{A relationship to fluid mechanics}
\label{sec:fluid}
As mentioned in the introduction,
we give an application of Theorem~\ref{thm:main2}
to fluid mechanics:
Consider a two-dimensional flow on the $xy$-plane
with velocity vector field
$\vect{v}=(u,v)$, and with density $\rho$ and pressure $p$.
We assume the following:
\begingroup
\renewcommand{\theenumi}{(\roman{enumi})}
\renewcommand{\labelenumi}{(\roman{enumi})}
\begin{enumerate}
 \item\label{item:ass-1}
      The fluid is {\em barotropic}, that is,
      there exists a strictly increasing function
      $p(s)$ ($s>0$) 
      such that the pressure $p$ is expressed by
      $p=p(\rho)$.
      A positive function $c$ defined by 
      \begin{equation}\label{eq:sound-speed}
          c^2 = \frac{dp}{d\rho}=p'(\rho)
      \end{equation}
      is called the {\em local speed of sound}, 
      cf.\ \cite[pages 5--6]{B}.
 \item\label{item:ass-2}
      The flow is steady, that is,
      $\vect{v}$, $p$ and $\rho$ do not depend on 
      time.
 \item\label{item:ass-3}
      There are no external forces,
 \item\label{item:ass-4}
      and the flow is irrotational,
      that is, $\rot\vect{v}(=v_x-u_y)=0$.
\end{enumerate}
\endgroup
By the assumption \ref{item:ass-2}, the equation of continuity
is reduced to
\[
   \div(\rho\vect{v})=(\rho u)_x + (\rho v)_y=0.
\]
Hence there exists locally a smooth function $\psi=\psi(x,y)$
such that
\begin{equation}\label{eq:stream}
  \psi_x = -\rho v,\qquad \psi_y=\rho u,
\end{equation}
which is called the {\em stream function\/} of the flow.
The following assertion is well-known:
\begin{fact}\label{prop:eq}
 The stream function $\psi$ of a two-dimensional 
 flow under the conditions \ref{item:ass-1}--\ref{item:ass-4}
 satisfies
 \begin{equation}\label{eq:stream-eq}
    (\rho^2c^2-\psi_y^2)\psi_{xx}+
    2\psi_x\psi_y\psi_{xy}+
    (\rho^2c^2-\psi_x^2)\psi_{yy}=0.
 \end{equation}
\end{fact}
\begin{proof}
 By the assumptions \ref{item:ass-2}, \ref{item:ass-2} and \ref{item:ass-4},
 Euler's equation of motion
 \[
     \frac{\partial \vect{v}}{\partial t}
       +
       u\frac{\partial\vect{v}}{\partial x}+v\frac{\partial\vect{v}}{\partial y}
       +\frac{1}{\rho}\grad p=0
 \]
 is reduced to
 \begin{equation}\label{eq:motion}
     uu_x+vv_x+\frac{p_x}{\rho}=0,\qquad
     uu_y+vv_y+\frac{p_y}{\rho}=0,
 \end{equation}
 that is, 
 \begin{equation}\label{eq:motion2}
     dp +\rho q\,dq = 0
     \qquad (q = |\vect{v}|=\sqrt{u^2+v^2}).
 \end{equation}
 Here, by the barotropicity \ref{item:ass-1}, we have
 \begin{equation}\label{eq:pot-dp}
    p_x = \frac{\partial}{\partial x}p(\rho)=
          p'(\rho)\frac{\partial\rho}{\partial x}
        = c^2\rho_x,\qquad
    p_y = c^2\rho_y.
 \end{equation}
 Substituting these into the equation of motion \eqref{eq:motion}, 
 we have
 \begin{equation}\label{eq:rho-x}
    \rho_x = \frac{p_x}{c^2}=-\frac{\rho}{c^2}(uu_x+vv_x),\qquad
    \rho_y=-\frac{\rho}{c^2}(uu_y+vv_y),
 \end{equation}
 and hence 
 \begin{align*}
  (\rho v)_x &=
      \rho_x v + \rho v_x
            =-\frac{\rho}{c^2}(uu_x+vv_x)v + \rho v_x,\\
  (\rho v)_y &=
      \rho_y v + \rho v_y
            =-\frac{\rho}{c^2}(uu_y+vv_y)v + \rho v_y,\\
  (\rho u)_x &=
      \rho_x u + \rho u_x
            =-\frac{\rho}{c^2}(uu_x+vv_x)u + \rho u_x,\\
  (\rho u)_y &=
      \rho_y u + \rho u_y
            =-\frac{\rho}{c^2}(uu_y+vv_y)u + \rho u_y
 \end{align*}
 hold.  Thus, we have
 \begin{align*}
  (\rho^2c^2-\psi_y^2)\psi_{xx}&=
  (\rho^2c^2-\rho^2 u^2)(-\rho v)_x
  =-\rho^2(c^2-u^2)
    \left(
      -\frac{\rho}{c^2}(uu_x+vv_x)v + \rho v_x
    \right),\\
  \psi_x\psi_y \psi_{xy}
    &= -\rho^2 u v (-\rho v)_y
    = \rho^2 u v \left(-\frac{\rho}{c^2}(uu_y+vv_y)v + \rho v_y\right),\\
  \psi_x\psi_y \psi_{yx}
    &= -\rho^2 u v (\rho u)_x
    = -\rho^2 u v \left(-\frac{\rho}{c^2}(uu_x+vv_x)u + \rho  u_x\right),\\
  (\rho^2c^2-\psi_x^2)\psi_{yy}&=
  (\rho^2c^2-\rho^2 v^2)(\rho u)_y
  =\rho^2(c^2-v^2)
    \left(
      -\frac{\rho}{c^2}(uu_y+vv_y)u + \rho u_y
    \right).
 \end{align*}
 Summing these up, it holds that 
 \begin{equation}\label{eq:concl2}
   (\rho^2c^2-\psi_y^2) \psi_{xx}
   + 2 \psi_x\psi_y \psi_{xy}
   + (\rho^2c^2-\psi_x^2)\psi_{yy}
   = \rho^3(u^2+v^2-c^2)(v_x-u_y).
 \end{equation}
 Here, by the assumption \ref{item:ass-4}, we have $v_x=u_y$.
Then we have the conclusion.
\end{proof}%
When $\rho c=1$, 
the equation \eqref{eq:stream-eq} coincides with the zero mean curvature
equation \eqref{zm}.
In this case, \eqref{eq:sound-speed} yields that
\begin{equation}\label{eq:ppp}
   p=p_0-\frac{1}{\rho},
\end{equation}
where $p_0$ is a positive constant.
This means that
the zero mean curvature equation
induces a virtual gas.
For actual gas,
the pressure $p$ is proportional to $\rho^\gamma$, 
where $\gamma$ is a constant
($>1$, $\gamma\approx 1.4$ for air).

Differentiating \eqref{eq:ppp}, we have
that
\begin{equation}\label{eq:pp}
    dp=\frac{d\rho}{\rho^2}.
\end{equation}
Substituting \eqref{eq:pp} into \eqref{eq:motion2},
we have that
\[
     d\left(-\frac1{\rho^2}+q^2 \right)=0,
\]
that is, there exists a constant $k$
such that
\begin{equation}\label{eq:int-euler}
       -\frac1{\rho^2}+q^2 =k.
\end{equation}
Here, by \eqref{eq:stream}, 
it holds that
\begin{equation}\label{eq:stream-vel}
   (u,v)=\frac{1}{\rho}(\psi_y,-\psi_x).
\end{equation}
Thus, \eqref{eq:int-euler} is equivalent to
\[
   1-\psi_x^2-\psi_y^2=-k \rho^2.
\]
We suppose that
$1-\psi_x^2-\psi_y^2$ does not vanish identically.
Then $k\ne 0$ and
we may set $\sqrt{|k|}=1/\rho_0$, 
where $\rho_0$ is a positive constant.
Then we have
\[
    \rho=\rho_0|1-\psi_x^2-\psi_y^2|^{1/2}.
\]
Note that by \eqref{eq:stream-vel} and 
the assumption $c\rho=1$, 
the speed $|\vect{v}|=q$ is greater 
than (resp.\ less than)
the speed of sound $c$, that is, the 
flow is {\it supersonic}
(resp.\ subsonic), if and only if 
$1-\psi_x^2-\psi_y^2<0$ (resp.\ $>0$).

Suppose now that the flow changes 
from being subsonic to supersonic at a curve
\[
    \sigma(t):=\bigl(x(t),y(t)\bigr)\qquad (a\le t\le b).
\]
Without loss of generality, we may assume that $t$ is an arclength 
parameter of the curve $\sigma$.
In particular,
\[
   \rho=\rho_0|1-\psi_x^2-\psi_y^2|^{1/2}=0
\]
holds on the curve $\sigma$.
Since the local speed of sound
is given by
\[
   c=\rho^{-1}
         =\rho_0^{-1}|1-\psi_x^2-\psi_y^2|^{-1/2},
\]
the curve $\sigma$ is a singularity of the flow,
although the stream function itself is real analytic near $\sigma$.
Moreover, we suppose that each point of the curve $\sigma(t)$
is a non-degenerate point of type change with respect to $\psi$.
Then, as seen in the proof of Proposition~\ref{prop:Kl2},
we can take the parameter $t$ of the curve $\sigma$
such that $t\mapsto (t,\sigma(t))$ gives a non-degenerate
null curve in $\R^3_1$. Moreover,
it holds that
\[
   x'(t)=\psi_x\bigl(x(t),y(t)\bigr),
    \qquad y'(t)=\psi_y\bigl(x(t),y(t)\bigr)
\]
and $x'(t)^2+y'(t)^2=1$.
Since
$\nabla B\ne 0$
at a non-degenerate point of type change,
the proof of Proposition \ref{eq:id} 
yields that $\sigma''(t)\ne 0$. 
Since $x'(t)^2+y'(t)^2=1$,
this implies that $\sigma(t)$ is a locally convex curve.
Consequently, we get the following assertion:
\begin{theorem}\label{cor:st}
 Let $\sigma(t):=\bigl(x(t),y(t)\bigr)$ be a 
locally convex curve on the $xy$-plane 
 with an arc-length parameter $t$.
 Then the graph $t=\psi(x,y)$ of the zero mean curvature surface
$\hat\phi_{\tilde \sigma}$ as in Theorem~\ref{thm:main2}
 associated to the null curve $\tilde\sigma:=\bigl(t,x(t),y(t)\bigr)$
 gives a real analytic stream function
 satisfying 
\eqref{eq:stream-eq} with \eqref{eq:ppp}
 {\rm(}i.e. $c\rho=1${\rm)} 
 which changes from being
subsonic to supersonic at the curve
 $\sigma$. 
 Moreover,
 \[
    (u,v):=\frac1{\rho}(\psi_y,-\psi_x) 
    \]
gives the velocity vector of the flow such that
 \begin{enumerate}
  \item $u^2+v^2$ diverges to $\infty$ 
   as $(x,y)$ approaches the image of $\sigma$. 
  \item The flow changes from being subsonic to 
  being supersonic across $\sigma$.
\item The acceleration vector $\sigma''(t)$ points 
to the supersonic region.
 \end{enumerate}
\end{theorem}

\begin{proof}
Since $\sigma$ is locally convex,
its lift $\tilde\sigma$ is a non-degenerate 
null curve in $\R^3_1$.
Then we can apply Theorem \ref{thm:main2}
for the curve $\tilde \sigma$, and get a graph 
$\psi(x,y)$ of a zero mean curvature surface
which changes type at $\tilde \sigma$.
Then $\psi$ can be considered as a stream function 
which changes from being subsonic to supersonic 
at the curve $\sigma$. 
We set
$$
P:=\frac{\sigma(t+s)+\sigma(t-s)}2
\approx \sigma(t)+\frac{\sigma''(t)}2s^2,
$$
which is the midpoint of the two points 
$\sigma(t+s),\sigma(t-s)$
of the curve $\sigma$.
By Theorem \ref{thm:main2},
the flow is supersonic at the point $P$, 
which proves the last assertion.
\end{proof}

\begin{acknowledgement}
The seventh author thanks Osamu Kobayashi  
 for a fruitful discussion at Osaka University,
who
suggested that our subject might be applied to
the theory of 2-dimensional compressible
gas flow.
\end{acknowledgement}

\end{document}